%% file: complexity_PL.tex
\documentclass {amsart}

\usepackage{amssymb} \usepackage{amsfonts} \usepackage{amsmath}
\usepackage{amsthm} \usepackage{epsfig} 
\usepackage{color}
\usepackage{amscd}
\usepackage[all]{xy}

\include{defs_complexity}

\author{Bruno Martelli}
\address{Dipartimento di Matematica ``Tonelli'', Largo Pontecorvo 5, 56127 Pisa, Italy}
\email{martelli at dm dot unipi dot it}
\title[Complexity of PL-manifolds]{Complexity of PL-manifolds}

\begin{document}

\begin{abstract}
We extend Matveev's complexity of $3$-manifolds to PL compact manifolds of arbitrary dimension, and we study its properties. The complexity of a manifold is the minimum number of vertices in a simple spine. We study how this quantity changes under the most common topological operations (handle additions, finite coverings, drilling and surgery of spheres, products, connected sums) and its relations with some geometric invariants (Gromov norm, spherical volume, volume entropy, systolic constant). 

Complexity distinguishes some homotopically equivalent manifolds and is positive on all closed aspherical manifolds (in particular, on manifolds with non-positive sectional curvature). There are finitely many closed hyperbolic manifolds of any given complexity. On the other hand, there are many closed 4-manifolds of complexity zero (manifolds without 3-handles, doubles of 2-handlebodies, infinitely many exotic K3 surfaces, symplectic manifolds with arbitrary fundamental group).

\end{abstract}

\maketitle

\section*{Introduction}
The \emph{complexity} $c(M)$ of a compact 3-manifold $M$ (possibly with boundary) was defined in a nice paper of Matveev \cite{Mat} as the minimum number of vertices of an \emph{almost simple spine} of $M$. In that paper he proved the following properties:
\begin{description}
\item[additivity] $c(M\# M')=c(M) + c(M')$ for any (boundary-)connected sum.
\item[finiteness] There are finitely many closed irreducible (or cusped hyperbolic) 3-manifolds of bounded complexity.
\item[monotonicity] If $M_S$ is obtained by cutting $M$ along an incompressible surface $S$, then $c(M_S)\leqslant c(M)$.
\end{description}
Thanks to the combinatorial nature of spines, it is not hard to classify all manifolds having increasing complexity $0,1,2,\ldots$ Tables have been produced in various contexts, see \cite{Bu, CaHiWe, FriMaPe3, survey, MaPe, Mat11, Mat:book} and the references therein (and Table \ref{census:table} below). Some of these classifications were actually done using the dual viewpoint of singular triangulations, which turns out to be equivalent to Matveev's for the most interesting 3-manifolds.

We extend here Matveev's complexity from dimension 3 to arbitrary dimension.
To do this, we need to choose an appropriate notion of spine. In another paper \cite{Mat:special} written in 1973, Matveev defined and studied \emph{simple spines} of manifolds in arbitrary dimension. A simple spine of a compact manifold is a (locally flat) codimension-1 subpolyhedron with generic singularities, onto which the manifold collapses. If the manifold is closed there cannot be any collapse at all and we therefore need to priorly remove one ball.

Simple spines are actually not flexible enough for defining a good complexity. In dimension 3, as an example, any simple spine for $S^3$ (or, equivalently, $D^3$) is a too complicated and unnatural object, such as Bing's house or the abalone. Every simple spine of $D^3$ has at least one vertex: however, a reasonable complexity must be zero on discs and spheres. 

To gain more flexibility, Matveev defined in 1988 the more general class of \emph{almost simple} spines \cite{Mat88, Mat} of 3-manifolds. An almost simple polyhedron is a compact polyhedron that can be locally embedded in a simple one. This more general definition allows one to use very natural objects as spines, such as a point for $D^3$ or a circle for $D^2\times S^1$: a point is \emph{not} a vertex by definition, and hence $c(D^3)=0$, as required. We show here that the notion of almost simple spine extends naturally to all dimensions.

This is in fact not the only way to gain more flexibility. A different possibility consists in enlarging the notion of spine by admitting an arbitrary number of open balls in its complement. Following that road, a 2-sphere is a simple spine of $S^3$ (or $D^3$) without vertices, and hence $c(D^3)=0$ again. One might also allow simultaneously almost simple polyhedra \emph{and} more balls in their complement.

In our attempt to define a suitable complexity in any dimension $n$, we are apparently forced to choose among three different definitions of complexity, and the choice seems only a matter of taste: as a spine for $S^n$, do we allow a point, an equator $(n-1)$-sphere, or both?

Luckily, these three definitions are actually equivalent and lead to the same complexity $c(M^n)$, in every dimension $n$. This non-trivial fact shows that $c(M^n)$ is indeed a very natural quantity to associate to a compact manifold $M^n$. For the sake of clarity, we choose in Section \ref{simple:section} the simplest definition, which takes simple polyhedra and admits more balls in their complement. The other definitions and the proof of their equivalence is deferred to Section \ref{alternative:section}.

Having settled the problem of defining $c$, we turn to studying its properties. Three-dimensional complexity is already widely studied, and 1- and 2-dimensional ones are quite boring, so in this introduction we focus mainly on dimension 4. 

We start by studying how $c$ varies when a handle is added. Quite surprisingly, complexity can always be controlled. If a 4-manifold $N$ is obtained from $M$ by adding a handle of index $i>0$, we have $c(N)\leqslant c(M)$, except when $i=3$: in that case we get the opposite inequality $c(N)\geqslant c(M)$. When $i=4$ we actually have $c(N)=c(M)$. 

These simple inequalities already allow to prove many things, and namely that plenty of 4-manifolds have complexity zero, in contrast with the 3-dimensional case. These include all 4-manifolds (with or without boundary) having a handle-decomposition without 3-handles, and all the doubles of 2-handlebodies (\emph{i.e.}~manifolds decomposing without 3- and 4-handles). The first set includes many simply connected manifolds (maybe all), the second set contains closed manifolds with arbitrary (finitely presented) fundamental group.

We can find more. It is easy to see that a non-trivial product $M^k\times N^{n-k}$ with boundary has a spine without vertices. Therefore every closed 4-manifold obtained from a non-trivial product by adding handles of index $\neq 3$ has complexity zero. Among manifolds that may constructed in this way, we find the infinitely many exotic K3 surfaces discovered by Fintushel and Stern in \cite{FiSt} and the closed symplectic manifolds with arbitrary fundamental group exhibited by Gompf in \cite{Go} (both types of manifolds are built by attaching handles of index $\neq 3$ to a product $M^3\times S^1$). 

As we have seen, there are plenty of 4-manifolds of complexity zero, although in many cases describing explicitly their spines is not obvious. One could guess that complexity is just zero on all 4-manifolds. Luckily, this is not the case. Various non-triviality results (in all dimensions $n$) are proved in this paper. 

A closed $n$-manifold $M$ with complexity zero must indeed fulfill some strict requirements. First of all, it cannot be aspherical. Moreover, its Gromov norm $||M||$ vanishes. If $\pi_1(M)$ is infinite and (virtually) torsion-free, some other geometric invariants of $M$ also vanish: the spherical volume $T(M)$ defined by Besson, Courtois, and Gallot in \cite{BCG}, the volume entropy $\lambda (M)$, and the systolic constant $\sigma(M)$, defined by Gromov in \cite{Gro2}. 

Concerning Gromov norm, we actually have $c(M)\geqslant ||M||$ for every closed aspherical manifold. This shows in particular that there are closed manifolds of arbitrarily high complexity in all dimensions. It also implies that there are finitely many closed hyperbolic $n$-manifolds of bounded complexity: this is a mild extension of the 3-dimensional finiteness property, proved by Matveev for all closed irreducible 3-manifolds.

The triviality and non-triviality results just stated suggest that $c(M)$ is a well-balanced quantity which could reasonably measure how ``complicate'' a manifold is. We hope that this new invariant will help to understand better the enormous set of PL (equivalently, smooth) compact 4-manifolds. 

\subsection*{Structure of the paper}
In Section \ref{main:section} we list all the properties of $c$ that are proved in this paper.
Some basic notions of piecewise-linear topology are collected in Section \ref{PL:section}. Simple spines and complexity are then introduced in Section \ref{simple:section}. Some of our definitions are somehow different from the ones given by Matveev: in Sections \ref{collars:section} and \ref{alternative:section} we prove that they are equivalent. 

In Section \ref{triangulations:section} we construct simple spines as objects dual to triangulations. In Section \ref{drilling:section} we show how to modify correspondingly a spine when the manifold is drilled along some sbpolyhedron. This basic operation will be used at many stages in the rest of the paper. 

In Section \ref{handles:section} we study how complexity changes under handle addition, sphere drilling, and connected sum. In Section \ref{coverings:section} we study the complexity of products and of finite coverings. In Section \ref{normal:section} we introduce a generalization of normal surfaces to arbitrary dimension and show how to ``cut'' a simple spine along a normal hypersurface.

In Section \ref{nerve:section} we study the \emph{nerve} of a simple spine $P$: the nerve is a simplicial complex determined by the stratification of $P$, which contains many informations on the topology of the manifold. The nerve is the key tool to prove various non-triviality results for $c$.
The relations between complexity and homotopy invariants, Gromov norm, and riemannian geometry are then studied in Sections \ref{homotopy:section}, \ref{norm:section}, and \ref{riemannian:section}. Finally, Section \ref{four:manifolds:section} is devoted to four-manifolds.

\subsection*{Acknowledgements}
We would like to thank Katya Pervova for suggesting improvements on an earlier version of the manuscript, and Roberto Frigerio for the many discussions on bounded cohomology and Gromov norm.

\section{Main results} \label{main:section}
We define the complexity $c(M)$ of any compact PL manifold $M$ in Section \ref{simple:section}. 
The definition is of course also applicable to every smooth compact manifold by taking its unique PL structure \cite{Wh}.
In this section we collect all the properties of $c$ proved in this paper.

\subsection*{Topological operations}
Simple spines are flexible. Most topological operations on manifolds can be translated into some corresponding modifications of their spines. Various estimates on the complexity are therefore proved by examining how the number of vertices may vary along these modifications.

We collect here some estimates. We start with products.

\begin{description}
\item[product with boundary] A product $N=M\times M'$ with non-empty boundary has $c(N)=0$.
\end{description}
In other words, if either $M$ or $M'$ is bounded, then $c(M\times M')=0$. If both $M$ and $M'$ are closed, we may have $c(M\times M')>0$: this holds for instance if both $M$ and $M'$ are aspherical (and so $N$ is), for instance if $M=M'=S^1$. On the other hand, we have the following.

\begin{description}
\item[sphere product] We have $c(M\times S^n)=0$ if $n\geqslant 2$.
\end{description}
Note that $||M\times S^n||=0$ for any $n\geqslant 1$. We are not aware of any general inequality relating $c(M)$, $c(M')$, and $c(M\times M')$ when both manifolds are closed. We turn to coverings.

\begin{description}
\item[covering] If $\widetilde M \to M$ is a degree-$d$ covering, then $c(\widetilde M)\leqslant dc(M)$.
\end{description}
In contrast with Gromov norm, this inequality is far from being an equality in general. For instance, lens spaces have arbitrarily high complexity while their universal covering $S^3$ has complexity zero.

We investigate the effect of adding a $i$-handle to a $n$-manifold $M$. Quite surprisingly, we always get a one-side estimate when $n\geqslant 4$, which depends only on the codimension $n-i$.

\begin{description}
\item[handles] Let $N$ be obtained from $M$ by adding a handle of index $i$. We have:
\begin{itemize}
\item $c(M)\leqslant c(N)$ if $i<n-1$ and $n\geqslant 3$,
\item $c(M)\geqslant c(N)$ if $i=n-1$ and $n\geqslant 4$,
\item $c(M)=c(N)$ if $i=n$.
\end{itemize}
\end{description}

These estimates imply a series of inequalities concerning connected sums and drilling along spheres of any dimension.

\begin{description}
\item[connected sum] We have $c(M^n\# N^n)\leqslant c(M^n)+c(N^n)$ for every (boundary) connected sum in dimension $n\geqslant 3$.
\end{description}
Matveev proved that an equality holds in dimension three \cite{Mat}. We do not know if it still holds in dimension $n\geqslant 4$.

We turn to sphere drilling. If $S\subset M$ is a submanifold, we denote by $M_S$ the manifold obtained by removing from $M$ an open regular neighborhood of $S$. As for handle addition, if $S$ is a $k$-sphere and $n\geqslant 4$ we get a one-side estimate which depends only on the dimension $k$.

\begin{description} 
\item[sphere drilling] Let $S\subset M$ be a $k$-sphere. We have
\begin{itemize}
\item $c(M_S)\leqslant c(M)$ if $k=1$ and $n\geqslant 4$,
\item $c(M_S)\geqslant c(M)$ if $k>1$.
\end{itemize}
\end{description}
The (PL-)sphere $S$ does not need to be locally flat. If $S$ has a product regular neighborhood $D^{n-k}\times S$ we can perform a \emph{surgery} by substituting this neighborhood with $S^{n-k-1}\times D^{k+1}$ along some map. If $k=1$, the previous result implies the following.

\begin{description}
\item[surgery] If $N$ is obtained from $M$ by surgery along a simple closed curve and $n\geqslant 4$, then $c(N)\leqslant c(M)$.
\end{description}
A strict inequality holds in some cases. 
\begin{description}
\item[strict inequality] If $M$ is closed with $c(M)>0$ and $n\geqslant 4$, there is a simple closed curve $\gamma\subset M$ such that $c(N)<c(M)$ if $N$ is obtained by drilling or surgery along $\gamma$.
\end{description}
This implies the following result. Very often in dimension 4 a complicate manifold becomes ``simpler'' after summing it with $S^2\times S^2$. The complexity might estimate this phenomenon as follows.

\begin{description}
\item[stabilization]
If $M$ is a simply connected closed 4-manifold with $c(M)>0$ then $c\big(M\#(S^2\times S^2)\big)<c(M)$.
\end{description}
However, we do not know if there exists any simply connected 4-manifold of positive complexity!

Finally, an important result of 3-dimensional complexity, due to Matveev \cite{Mat}, says that $c(M_S)\leqslant c(M)$ whenever $S$ is an incompressible surface. Unfortunately, the notion of incompressibility does not extend appropriately to higher dimensions. Having in mind that every class in $H_2(M^3,\matZ)$ is represented by an incompressible surface, we extend a weaker version of Matveev's result as follows.

\begin{description}
\item[hypersurfaces] Every class in $H_{n-1}(M,\matZ_2)$ is represented by a hypersurface $S$ such that $c(M_S)\leqslant c(M)$.
\end{description}

This result is proved by extending the 3-dimensional notion of normal surface to any dimension: this extension might be of independent interest. 

The results just stated are proved in Sections \ref{handles:section}, \ref{coverings:section}, and \ref{normal:section}.

\subsection*{Gromov norm and triangulations}
Let $||M||$ and $t(M)$ be respectively the Gromov norm \cite{Gro} and the minimum number of simplexes in a triangulation of $M$.
\begin{description}
\item[Gromov norm (1)] If $M$ is closed with virtually torsion-free $\pi_1(M)$, then 
$$||M||\leqslant c(M)\leqslant t(M).$$
\end{description}
Note that if $M$ is aspherical then $\pi_1(M)$ is torsion-free and hence the inequalities hold for any aspherical manifold $M$. Actually, only the left inequality requires this hypothesis on $\pi_1(M)$, and we do not know if it is really necessary. Both inequalities might be justified informally by saying that simples spines are more flexible than triangulations, but not as flexible as real homology cycles. 

The above result can be strengthened in complexity zero, by dropping the hypothesis on $\pi_1(M)$ and admitting amenable boundary. The boundary $\partial M$ is amenable if the image of every connected component of $\partial M$ in $\pi_1(M)$ is an amenable group.
\begin{description}
\item[Gromov norm (2)] Let $M$ be a manifold with (possibly empty) amenable boundary. If $c(M)=0$ then $||M||=0$.
\end{description}
The amenability hypothesis is necessary, since a genus-2 handlebody has complexity zero and positive Gromov norm. 

The results just stated are proved in Sections \ref{triangulations:section} and \ref{norm:section}.

\subsection*{Homotopy type}
It might be that every simply connected manifold has complexity zero. This question is open only in dimension 4.
\begin{description}
\item [simply connected] Every simply connected manifold of dimension $\neq 4$ has complexity zero.
\end{description}
On the other hand, we have the following.
\begin{description}
\item [arbitrary fundamental group] Every finitely presented group is the fundamental group of a closed 4-manifold with complexity zero.
\end{description}
Complexity detects aspherical manifolds, in some sense.
\begin{description}
\item [aspherical manifolds] If $M$ is closed aspherical, then $c(M)>0$.
\end{description}
This shows in particular that complexity behaves quite differently from Gromov norm. For instance, complexity detects non-positive curvature, while Gromov norm detects negative curvature: the $n$-torus $T^n$ has $c(T)>0$ and $||T||=0$. 

We also note that complexity is \emph{not} a homotopy invariant, since it distinguishes some homotopically equivalent lens spaces: we have $c(L_{7,1}) = 4$ and $c(L_{7,2})=2$, see \cite{Mat}. We do not know if it distinguishes different PL manifolds sharing the same topological structure.

The results just stated are proved in Section \ref{homotopy:section}.

\subsection*{Riemannian geometry}
We compare the complexity of a smooth manifold $M$ with other invariants coming from riemannian geometry. A relation between the volume of a riemannian manifold and its complexity can be given by bounding both the sectional curvature and the injectivity radius. The second inequality in the following result is due to Gromov \cite{Gro}.

\begin{description}
\item [volume] Let $M^n$ be a riemannian manifold with everywhere bounded sectional curvature $|K(M)|\leqslant 1$. Then
$$c(M)\leqslant t(M) \leqslant {\rm const}_n \frac {{\rm Vol}(M)}{{\rm inj}_* (M)^n}.$$
\end{description}
Here ${\rm inj}(M)$ is the injectivity radius, ${\rm inj}_*(M) = \min\{{\rm inj}(M),1\}$, and ${\rm const}_n$ is a constant depending on $n$. The same formula holds for Gromov norm $||M||$: in that case however the factor ${\rm inj}_*(M)^{-n}$ can be removed when $\pi_1(M)$ is residually finite \cite{Gro}. It is not possible to remove this factor here, since there are infinitely many hyperbolic 3-manifolds with bounded volume, while only finitely many can have bounded complexity. This holds in fact in all dimensions.

\begin{description}
\item [finiteness] There are finitely many closed hyperbolic $n$-manifolds of bounded complexity, for every $n$.
\end{description}

We do not know if the finiteness property can be extended to manifolds of non-negative curvature, or more generally to aspherical manifolds. As far as we know, it might also hold for elliptic manifolds. The results on Gromov norm allow to prove also the following.

\begin{description}
\item[cusped hyperbolic manifolds] Let $M$ be a compact manifold whose interior admits a complete hyperbolic metric of finite volume. Then $c(M)>0$.
\end{description}
This result is sharp since the Gieseking 3-manifold has complexity 1, see \cite{CaHiWe}.
Complexity is also related to other geometric invariants. A nice chain of inequalities, taken from \cite{Ko, PaPe}, holds for every closed orientable manifold $M$:
$$\frac{n^{n/2}}{n!} ||M||\leqslant 2^nn^{n/2}T(M)\leqslant \lambda(M)^n \leqslant h(M)^n \leqslant (n-1){\rm MinVol}(M).$$
From left to right, we find Gromov norm $||M||$, the spherical volume $T(M)$ defined by Besson, Courtois, and Gallot in \cite{BCG}, the volume entropy $\lambda (M)$, the topological entropy $h(M)$, and the minimum volume ${\rm MinVol}(M)$ defined by Gromov in \cite{Gro}. Another interesting invariant is the systolic constant $\sigma(M)$, defined by Gromov in \cite{Gro2}. We have  the following.

\begin{description}
\item[geometric invariants] Let $M$ be a closed orientable manifold with virtually torsion-free infinite fundamental group. If $c(M)=0$ then 
$$T(M)=\lambda(M)=\sigma(M)=0.$$
\end{description}
We do not know if the same hypothesis implies also $h(M)=0$. It does not imply ${\rm MinVol}(M)=0$ (for instance, if $M = (T^2\times S^2)\#\matCP^2$ we have $c(M)=0$ and ${\rm MinVol}(M)\geqslant {\rm const}_n |\chi(M)|>0$). 

Finally, we quote a result of Alexander and Bishop \cite{AleBi, AleBi2} relating the complexity and the width of a riemannian manifold with boundary. 
\begin{description}
\item[thin manifolds] There are some constants $a_2<a_3<\ldots $ such that if a riemannian manifold $M^n$ with boundary has (curvature-normalized) inradius less than $a_n$, then $c(M^n)=0$.
\end{description}

The results just stated are proved in Section \ref{riemannian:section}.

\subsection*{Low dimensions}
The complexity of manifolds of dimension 1 and 2 is easily calculated.
Concerning $1$-manifolds, we have $c(S^1)=c(D^1)=1$. Turning to dimension $2$, the complexity of a (compact) surface $\Sigma$ turns out to be as follows: 
\begin{itemize}
\item $c(\Sigma) = \max\{2-2\chi(\Sigma),0\}$ if $\Sigma$ is closed,
\item $c(\Sigma) = \max\{-2\chi(\Sigma),0\}$ if  $\Sigma$ has boundary.
\end{itemize}
The compact surfaces having complexity zero are $S^2, \matRP^2$, the annulus, and the M\"obius strip. The torus and the pair-of-pants have complexity $2$. 

The complexity of $3$-manifolds has been widely studied. Manifolds of low complexity have been listed via computer by various authors \cite{Bu, CaHiWe, FriMaPe3, survey, MaPe, Mat11, Mat:book}: the closed orientable irreducible ones are collected in Table \ref{census:table} according to their geometry.
The closed irreducible manifolds having complexity zero are $S^3$, $\matRP^3$, and $L_{3,1}$.

\begin{table}
  \begin{center}
    \begin{tabular}{rcccccccccccc}
      $c$ & $0$ & $1$ & $2$ & $3$ & $4$ & $5$ & $6$ & $7$ & $8$ & $9$ & $10$ & $11$ \\
      \hline
      lens spaces &
      $3$ & $2$ & $3$ & $6$ & $10$ & $20$ & $36$ & $72$ & $136$ & $272$ & $528$ & $1056$ \\

      other elliptic &
      . & . & $1$ & $1$ & $4$ & $11$ & $25$ & $45$ & $78$ & $142$ & $270$ & $526$ \\

      flat &
      . & . & . & . & . & . & $6$ & .    & .    & .      & .	& . \\
      
      Nil &
      . & . & . & . & . & . & $7$ & $10$ & $14$ & $15$   & $15$ & $15$ \\

      ${\rm SL}_2\matR$ &
      . & . & . & . & . & . & .   & $39$ & $162$ & $513$ & $1416$ & $3696$ \\
      
      Sol &
      . & . & . & . & . & . & .   & $5$  & $9$   & $23$  & $39$ & $83$ \\

      $\matH^2\times\matR$ &
      . & . & . & . & . & . & .   & .    & $2$   & .     & $8$ & $4$ \\
      
      hyperbolic &
      . & . & . & . & . & . & .   & .    & .     & $4$   & $25$ & $120$ \\
      
      not geometric &
      . & . & . & . & . & . & .   & $4$  & $35$  & $185$ & $777$ & $2921$ \\
      
      total  &
      $\bf{3}$ & $\bf{2}$ & $\bf{4}$ & $\bf{7}$ & $\bf{14}$ & $\bf{31}$ & $\bf{74}$ & $\bf{175}$ & $\bf{436}$ & $\bf{1154}$ & $\bf{3078}$ & $\bf{8421}$ \\

    \end{tabular}
  \end{center}
\caption{The number of irreducible orientable $3$-manifolds of complexity $c\leqslant 11$ in each geometry. The non-geometric manifolds decompose into geometric pieces according to their JSJ decomposition along tori (they are all graph manifolds when $c\leqslant 10$).}
\label{census:table}
\end{table}

We now devote our attention to dimension $4$. We start by studying the set of 4-manifolds of complexity zero. We describe here some interesting classes of such manifolds. These classes seem however far to exhaust the set of all 4-manifolds with complexity zero.

The various results stated above show that the set of all 4-manifolds of complexity zero contains all products $N\times N'$ with non-empty boundary or $N\in\{S^2,S^3\}$, and is closed under connected sums, finite coverings, addition of handles of index $\neq 3$, and drilling (or surgery) along simple closed curves. All the examples presented here are of this kind. We concentrate on closed manifolds for simplicity.

\begin{description}
\item[no 3-handles] Every closed 4-manifold that has a handle decomposition without 3-handles has complexity zero.
\end{description}

Every such manifold is necessarily simply connected. However, for many simply connected manifolds a decomposition without 3-handles does not seem to be known. Among these, we find the exotic K3 surfaces constructed by Fintushel and Stern in \cite{FiSt}. In fact, these manifolds are constructed by attaching handles of index $\neq 3$ to a product $M^3\times S^1$. Therefore we have the following.

\begin{description}
\item[exotic K3] The (infinitely many) exotic K3 surfaces $X_K$ constructed via Fintushel and Stern's \emph{knot construction} \cite{FiSt} from a knot $K\subset S^3$ have complexity zero.
\end{description}

We now introduce two different classes of closed 4-manifolds with arbitrary (finitely presented) fundamental group. Let a \emph{2-handlebody} be a 4-manifold which has a decomposition with 0-, 1-, and 2-handles.

\begin{description}
\item[doubles of 2-handlebodies] The double of any orientable 2-handlebody has complexity zero.
\end{description}
These manifolds have complexity zero because they are obtained by surgerying $(S^1\times S^3)\#\ldots \#(S^1\times S^3)$ along some curves. Every finitely presented group is the fundamental group of a 2-handlebody, which is in turn isomorphic to the fundamental group of its double. It is not true that \emph{any} double has complexity zero, because a double can be aspherical (for instance, a product of surfaces).

Another class was constructed by Gompf in \cite{Go}, in order to show that symplectic 4-manifolds may have arbitrary fundamental group. As above, these manifolds are constructed by attaching handles of index $\neq 3$ to a product $M^3\times S^1$, so we have the following.

\begin{description}
\item[symplectic manifolds] The closed symplectic manifolds with arbitrary fundamental group constructed by Gompf in \cite{Go} have complexity zero.
\end{description}

The results just stated are proved in Section \ref{four:manifolds:section}.

\section{Piecewise-linear topology} \label{PL:section}
We collect here the informations on piecewise-linear topology that
we will need. The basic definitions and tools are listed in Sections \ref{basic:definitions:subsection} and \ref{basic:tools:subsection}. More material can be found in~\cite{PL}. 
The notion of intrinsic stratification is taken from \cite{Ak, Arm, Mc, Sto} and described in Section \ref{intrinsic:subsection}. Stein factorization (which we take from \cite{CoTh}) is introduced in Section \ref{Stein:subsection}. Finally, in Section \ref{nerve:subsection} we define the \emph{nerve} of a pair $(X,Y)$ of polyhedra: this definition is original and might be of independent interest. 

The material contained in Sections \ref{Stein:subsection} and \ref{nerve:subsection} is used in Section \ref{nerve:section} to define the nerve of a pair $(M,P)$ when $P$ is a simple spine of $M$.

\subsection{Basic definitions} \label{basic:definitions:subsection}
\subsubsection{Simplicial complexes}
A (finite and abstract) \emph{simplicial complex} $K$ is a set of nonempty subsets of a given finite set $V(K)$ (the \emph{vertices} of $K$), such that $\{v\}\in K$ for all $v\in V(K)$ and if $\sigma \in K$ and $\tau\subset\sigma$ then $\tau\in K$. An element of $K$ is a \emph{face}. A \emph{subcomplex} is a subset of $K$ which is a complex. If $K$ and $L$ are simplicial complexes, a \emph{simplicial map} $f:K\to L$ is a function $f:V(K)\to V(L)$ such that if $\sigma\in K$ then $f(\sigma)\in L$.

\subsubsection{Triangulations}
A (finite) simplicial complex $K$ induces a compact topological space $|K|$, defined by taking a standard euclidean simplex for each element of $K$ and identifying them according to the face relations. A \emph{triangulation} of a compact topological space $X$ is a simplicial complex $K$ and a homeomorphism $f:|K|\to X$. Another triangulation $(L,g)$ of $X$ is a \emph{subdivision} of $(K,f)$ if the image of every simplex of $L$ is contained as a straight simplex in some simplex of $K$. Two triangulations of $X$ are \emph{related} if they have a common subdivision.

We will use the letter $T$ to indicate a triangulation, \emph{i.e.}~a pair $(K,f)$.

\subsubsection{Polyhedra}
A \emph{compact polyhedron} is a compact topological space $X$ equipped with a maximal family of related triangulations. A \emph{subpolyhedron} $X'\subset X$ is a subset which is the image of a subcomplex of some triangulation of $X$. If $X$ is a polyhedron containing compact polyhedra $X_1,\ldots, X_k$, a triangulation $K$ of $(X,X_1,\ldots,X_k)$ is a triangulation of $X$ where each $X_i$ is represented by some subcomplex; such a triangulation can be found by taking a common subdivision of the triangulations realizing $X_i$ as a subcomplex.

The standard $n$-simplex $\Delta^n$ is a polyhedron. We define the $n$-disc $D^n$ and $(n-1)$-sphere $S^{n-1}$ respectively as $\Delta^n$ and $\partial\Delta^n$. 

\subsubsection{Manifolds and maps}
A simplicial map $f:K\to L$ induces a continuous map $f:|K|\to |L|$. A map between polyhedra is \emph{piecewise-linear} (shortly, PL) if it is induced by a simplicial map on some triangulations. A polyhedron is a \emph{PL-manifold} (with boundary) if it is locally PL-homeomorphic to some point in $S^n$ ($D^n$). Every manifold and map mentioned in this paper is tacitly assumed to be PL. 

\subsection{Basic tools} \label{basic:tools:subsection}

\subsubsection{Derived complexes} \label{derived:subsubsection}
A simplicial complex $K$ defines a partially ordered set (briefly, a \emph{poset}) $i(K) = (K,\subseteq )$, the set of faces with their face relations. Conversely, a poset $(A,\leqslant)$ defines a simplicial complex $\eta(A,\leqslant)$, whose vertices are the elements of $A$, and whose faces are all finite subsets $\{a_0,\ldots, a_i\}$ such that $a_0<\ldots <a_i$. The simplicial complex $\eta(A,\leqslant)$ is the \emph{nerve} of $(A,\leqslant)$.

The simplicial complex $K' = \eta\circ i(K)$ is the \emph{derived} simplicial complex of $K$. Vertices of $K'$ correspond to faces of $K$. A simplicial map $f: K \to L$ induces an order-preserving map $i(K)\to i(L)$ and hence a \emph{derived} simplicial map $f':K'\to L'$. 

A triangulation $T = (K,f)$ of a space $X$ determines a \emph{barycentric subdivision} $T'=(K',f')$ of $X$, obtained by composing $f$ with the homeomorphism $|K'|\to |K|$ which sends every vertex of $K'$ to the barycenter of the corresponding face of $K$ (and is extended linearly on the rest of $|K'|$)

\subsubsection{Join, cone, and suspension}
The \emph{join} $K*L$ of two simplicial complexes $K$ and $L$ (with disjoint vertices) is the complex with vertices $V(K*L)=V(K)\cup V(L)$ and with faces $K\cup L \cup\{\sigma\cup\tau |\sigma\in K, \tau\in L\}$. The polyhedron $|K*L|$ depends only on $|K|$ and $|L|$ (up to homeomorphism) and can thus be denoted by $|K|*|L|$.

The \emph{cone} and \emph{suspension} of a polyhedron $P$ are respectively $C(P) = P* D^0$ and $\Sigma(P) = P*S^0$. We have $\Sigma^k(P) \equiv P*S^{k-1}$.

\subsubsection{Link, star, and regular neighborhood}
Let $K$ be a simplicial complex and $L\subset K$ a subcomplex. The \emph{star} $\st(L,K)$ of $L$ in $K$ is the minimal subcomplex of $K$ containing all faces that intersect some face of $L$. The \emph{link} $\lk(L,K)$ is the subcomplex of $\st(L,K)$ consisting of all faces not intersecting any face of $L$. 

When $Y\subset X$ are polyhedra and $T$ is a triangulation of $(X,Y)$, we indicate by $\lk(Y,T)$ and $\st(Y,T)$ the corresponding subpolyhedra of $X$. When $Y=\{y\}$ is a point, these polyhedra depend (up to homeomorphism) only on $y$ and not on $T$. 

In general, if $T$ is sufficiently subdivided, the star $\st(Y,T)$ does not depend on $T$ up to an isotopy in $X$ keeping $Y$ fixed: for instance, this holds after two barycentric subdivisions. In that case, the polyhedron $\st(Y,T)$ is the \emph{regular neighbhorhood} of $Y$ in $X$, which we denote by $R(Y)$.

When $X$ is a manifold, the regular neighborhood $R(Y)$ is a manifold with boundary.

\subsubsection{Collapse}
Let $K$ be a simplicial complex. Let $\sigma\in K$ be a face which is
properly contained in a unique face $\eta$. The subcomplex
$L=K\setminus\{\sigma,\eta\}$ is obtained from $L$ by an
\emph{elementary collapse}. 

Let $Y\subset X$ be any polyhedra. The polyhedron $Y$ is obtained from
$X$ via a \emph{elementary collapse} if it is so on some
triangulation. More generally, a \emph{collapse} of $X$ onto a
subpolyhedron $Z$ is a combination of
finitely many simplicial collapses.

\subsection{Intrinsic strata} \label{intrinsic:subsection}
We recall the notions of intrinsic dimension and strata of polyhedra,
see \cite{Ak, Arm, Mc, Sto}.

Let $Y\subset X$ be any polyhedra and $x\in X$ a point. The
\emph{intrinsic dimension} $d(x;X,Y)$ of the pair $(X,Y)$ at $x$ is
the maximum number $t$ such that 
\begin{enumerate}
\item there is a triangulation of $(X,Y)$ 
with $x$ contained in the interior of a $t$-simplex. 
\end{enumerate}
If $x\in Y$, this
is equivalent to each of the following conditions:

\begin{enumerate}
\addtocounter{enumi}{1}
\item the link of $x$ in $(X,Y)$ is the $t$-th suspension $\Sigma^t(W,Z)$ of some pair $(W,Z)$;
\item the star of $x$ in $(X,Y)$ is homeomorphic to $C(W,Z)\times D^t$
  with $x$ sent to $v\times c$, where $v$ is the vertex of the cone
  $C(W,Z)$ and $c\in\interior{D^t}$.
\end{enumerate}

The absolute notion of intrinsic dimension of a point $x$ in a
polyhedron $Y$ is defined as $d(x;Y)=d(x;Y,Y)$. If $x\not\in Y$ we have
$d(x;X,Y)=d(x;X)$. If $x\in Y$ we have $d(x;X,Y)\leqslant\{d(x,X),
d(x,Y)\}$.

A subpolyhedron $Y\subset X$ in a manifold $X$ is \emph{locally
  unknotted} at $x$ if $d(x;X,Y) = d(x;Y)$. When $Y$ is a manifold,
  this is equivalent with the standard notion of local flatness. 
The subpolyhedron $Y\subset X$ is locally unknotted if it is so at every $x\in Y$.

The intrinsic dimension can be easily calculated using the following
nice result of Armstrong and Morton \cite{Arm, Mor}:
\label{Check Morton}
\begin{prop}[Armstrong-Morton] \label{suspension:prop}
If the link of $x$ in $(X,Y)$ is the $t$-th suspension of some pair $(W,Z)$, and $(W,Z)$ is not itself a suspension, then $t=d(x;X,Y)$ and $(W,Z)$ is uniquely determined by $x$. 
\end{prop}
This easily implies the following.
\begin{ex} \label{unknotted:ex}
If $Y\subset X$ is locally unknotted, then $\lk(x,Y)\subset \lk(x,X)$ is locally unknotted for every $x\in Y$.
\end{ex}

The intrinsic dimension induces an intrinsic stratification of any
pair $(X,Y)$. The points of intrinisic dimension $k$ in $Y$ form the
\emph{$k$-stratum} of $Y$. The $k$-stratum is an (open) $k$-dimensional manifold made of finitely many connected components, called \emph{$k$-components} (or simply \emph{components}). Points in a $k$-component are all homogeneous, \emph{i.e.}~there is an ambient isotopy of $Y$ sending a point to any other. In particular, they have the same link.

The union of all points of intrinsic dimension $\leqslant k$ is the \emph{$k$-skeleton}: it is a $k$-dimensional polyhedron.  

\subsection{Stein factorization} \label{Stein:subsection}
A \emph{Stein factorization} of a (piecewise-linear) map $f:X\to Y$
between (compact) polyhedra is a decomposition $f=g\circ h$ into two
maps
$$
\begin{CD} 
X @>h>> Z @>g>> Y
\end{CD}
$$
such that $h$ has connected fibers and $g$ is finite-to-one. Every $f$
has a unique Stein factorization: the map $h$ is the quotient onto the
space $Z$ of connected components of the fibers of $f$, see Fig~\ref{Stein:fig}. We learned about this notion from \cite{CoTh}.

\begin{figure}
 \begin{center}
  \includegraphics[width = 9cm]{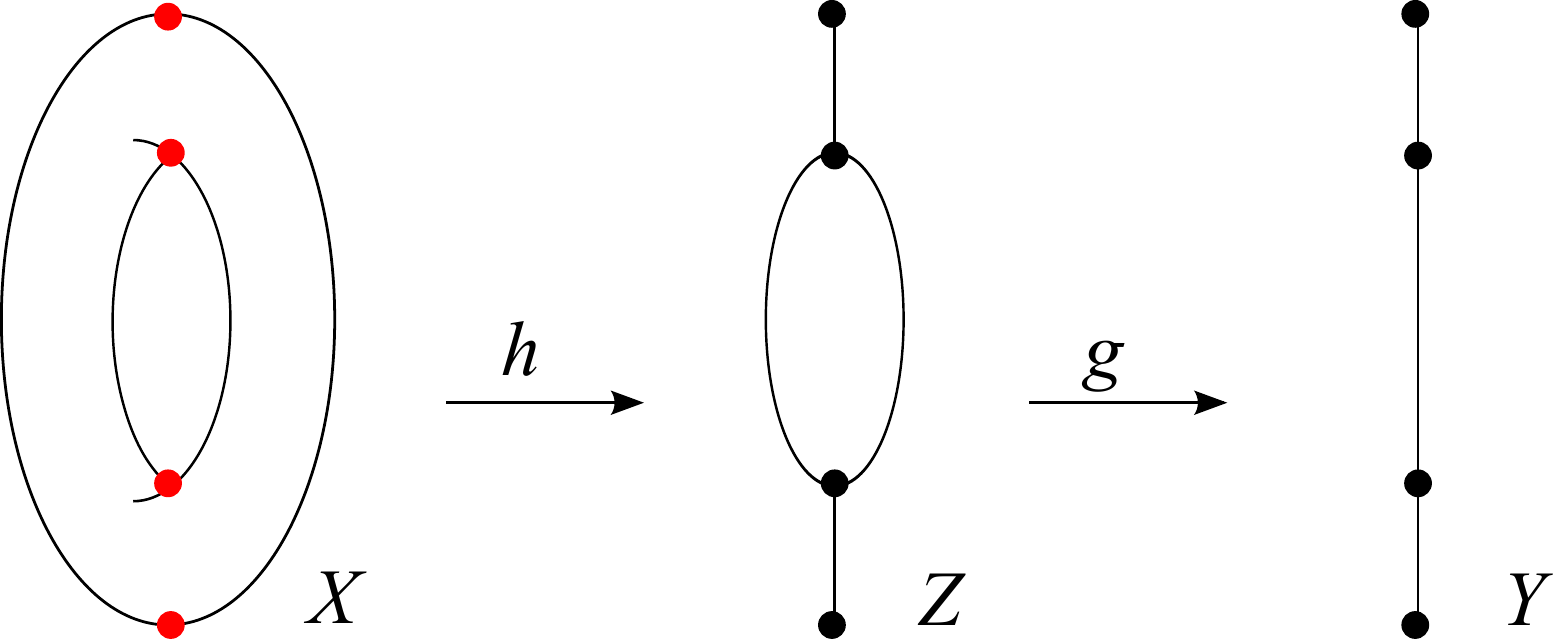}
 \end{center}
 \caption{The Stein factorization of a map.}
 \label{Stein:fig}
\end{figure}

We define the Stein factorization in the category of simplicial
complexes. Let $f:K\to L$ be a simplicial map. Let $f':K'\to L'$ be
its derived map. We define an intermediate simplicial complex $H$ as
follows. Consider the map $f':|K'|\to |L'|$. The vertices of $H$ are
the connected components of $(f')^{-1}(v)$ when $v$ varies among the
vertices of $L'$. The map $f':V(K')\to V(L')$ naturally splits along
two maps
$$
\begin{CD} 
V(K') @>h>> V(H) @>g>> V(L').
\end{CD}
$$
We now define a simplex in $H$ to be the image of any simplex in $K'$
along $h$. The resulting maps
$$
\begin{CD} 
K' @>h>> H @>g>> L'
\end{CD}
$$
are simplicial and $f'=g\circ f$. Since we used the derived map $f'$,
the map $h:|K'|\to |H|$ has indeed connected fibers everywhere (not
only at the vertices of $H$). The map $g:|H|\to|L'|$ is finite-to-one:
this is equivalent to the condition that $\dim g(\sigma) =\dim \sigma$
for every simplex $\sigma$ of $H$.

\subsection{Nerve} \label{nerve:subsection}
The nerve of a polyhedron is a simplicial complex which encodes the
incidences between its components, see Section
\ref{intrinsic:subsection}. 
We define it for pairs $(X,Y)$.

Let $Y\subset X$ be any polyhedra. The components of $(X,Y)$ form a
partially ordered set $(\calC,\leqslant)$: we set $C\leqslant C'$ if
$C\subset \overline{C'}$. If $C<C'$ then $\dim C < \dim C'$. The
\emph{pre-nerve} of $(X,Y)$ is the nerve $\calN_0 =
\eta(\calC,\leqslant)$ of this partially ordered set, see Section
\ref{derived:subsubsection} above.

Let $T$ be a triangulation of $(X,Y)$. If $T$ is sufficiently
subdivided, by sending every vertex of $T$ to the component to which
it belongs we get a surjective simplicial map $\varphi_0:T\to
\calN_0$, called the \emph{pre-nerve map}. It induces a surjective
continuous map $\varphi_0:X\to |\calN_0|$.

The pre-nerve map does not necessarily have connected fibers, so we
prefer to consider its Stein factorization, see Section
\ref{Stein:subsection}. The \emph{nerve} of $(X,Y)$ is the complex
$\calN$ obtained via the Stein factorization

$$
\begin{CD} 
T' @>\varphi>> \calN @>g>> \calN_0'
\end{CD}
$$
of the pre-nerve map $\varphi_0' = g\circ \varphi$. The map
$\varphi:T'\to \calN$ is the \emph{nerve map}. More generally, a
\emph{nerve map} is a map $\varphi:X\to |\calN|$ induced by some
(sufficiently subdivided) triangulation of $(X,Y)$.

\begin{ex} \label{nerve:ex}
The pre-nerve of $(X,Y) = (S^1,\{pt\})$ is a segment, while the nerve is a circle.
\end{ex}

\section{Simple polyhedra} \label{simple:section}
The definition of simple polyhedra in arbitrary dimensions is due to Matveev \cite{Mat:special}. We use it in Section \ref{complexity:subsection} to define the complexity of a manifold. This definition extends Matveev's complexity of 3-manifolds \cite{Mat}.

\subsection{The local model} \label{local:subsection}
\begin{figure} 
 \begin{center}
 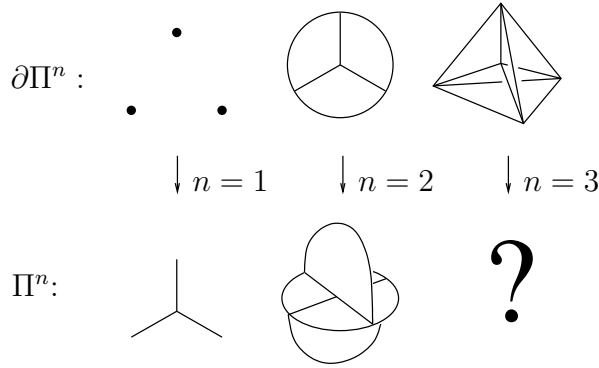
  \end{center} 
 \caption{The $(n-1)$-skeleton of the $(n+1)$-simplex and the cone
 $\Pi^n$ over it. The three-dimensional $\Pi^3$ is not drawn.}
 \label{model:fig}
\end{figure}

Let $\Delta=\Delta^{n+1}$ be the $(n+1)$-simplex. Let $\Pi^n$ be the cone over the $(n-1)$-skeleton of $\Delta$. The base of the cone is its  \emph{boundary} $\partial \Pi^n$, while $\interior{\Pi^n} = \Pi^n\setminus\partial\Pi^n$ is its \emph{interior}. Some examples are shown in Fig. \ref{model:fig}.

\begin{figure}
 \begin{center}
  \includegraphics[width = 7cm]{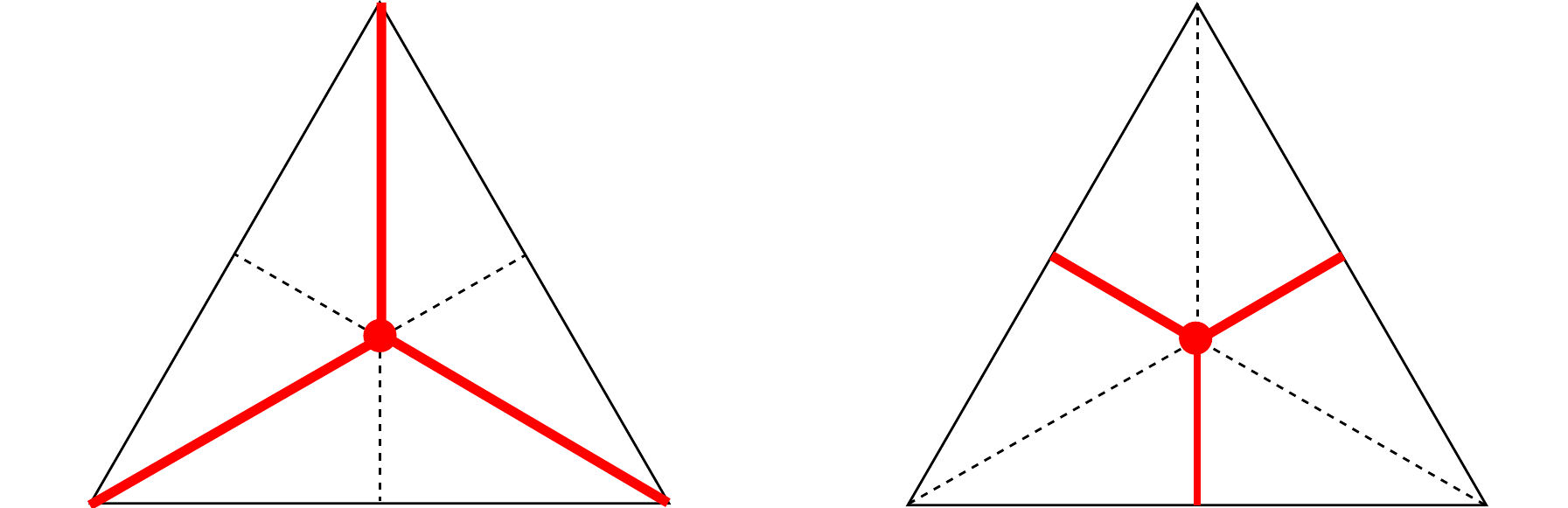}
 \end{center}
 \caption{The standard and dual representation of $\Pi^n$ inside $\Delta$. They are both subcomplexes of $\Delta'$. Here, $n=1$.}
 \label{representations:fig}
\end{figure}

\begin{figure}
 \begin{center}
  \includegraphics[width = 4cm]{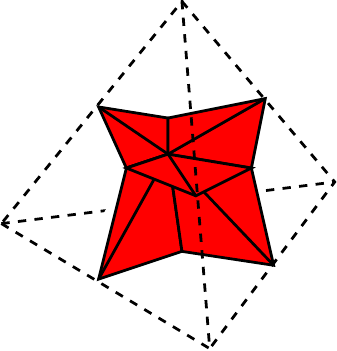}
 \end{center}
 \caption{The dual representation of $\Pi^2$ inside the tetrahedron $\Delta^3$.}
 \label{dualspine:fig}
\end{figure}

There are two representations of $\Pi^n$ inside $\Delta$, shown in
Fig~\ref{representations:fig}: the \emph{standard} and \emph{dual}
representation. They both describe $\Pi^n$ as a subcomplex of the
barycentric subdivision $\Delta'$. See also Fig.~\ref{dualspine:fig}. Both representations induce the same pair $(D^{n+1}, \Pi^n)$ up to homeomorphism. The dual description is investigated below in Section \ref{models:subsection}.

We define $\Pi^n_k$ as $\Pi^n_k=\Pi^{n-k}\times D^k$. The pair
$(D^{n+1}, \Pi^n_k) = D^k\times (D^{n-k+1}, \Pi^{n-k})$ is
well-defined up to homeomorphism. The \emph{boundary} $\partial\Pi^n_k
= \Pi^n_k\cap S^{n}$ is homeomorphic to the $k$-th suspension
$\Sigma^k(\partial\Pi^{n-k})$. Following Matveev, a point $x$ in
a polyhedron $P$ is \emph{of type} $k$ if its link is homeomorphic to
$\partial\Pi^n_k$ (and hence its star is homeomorphic to $\Pi^n_k$)
\footnote{Actually, our $\Pi^n_k$ corresponds to Matveev's
  $\Pi^n_{n-k}$: 
we prefer to define the type of a point coherently with 
Armstrong's general notion of intrinsic dimension.}. See Fig.~\ref{models:fig}.

\begin{figure}
 \begin{center}
  \includegraphics[width = 11cm]{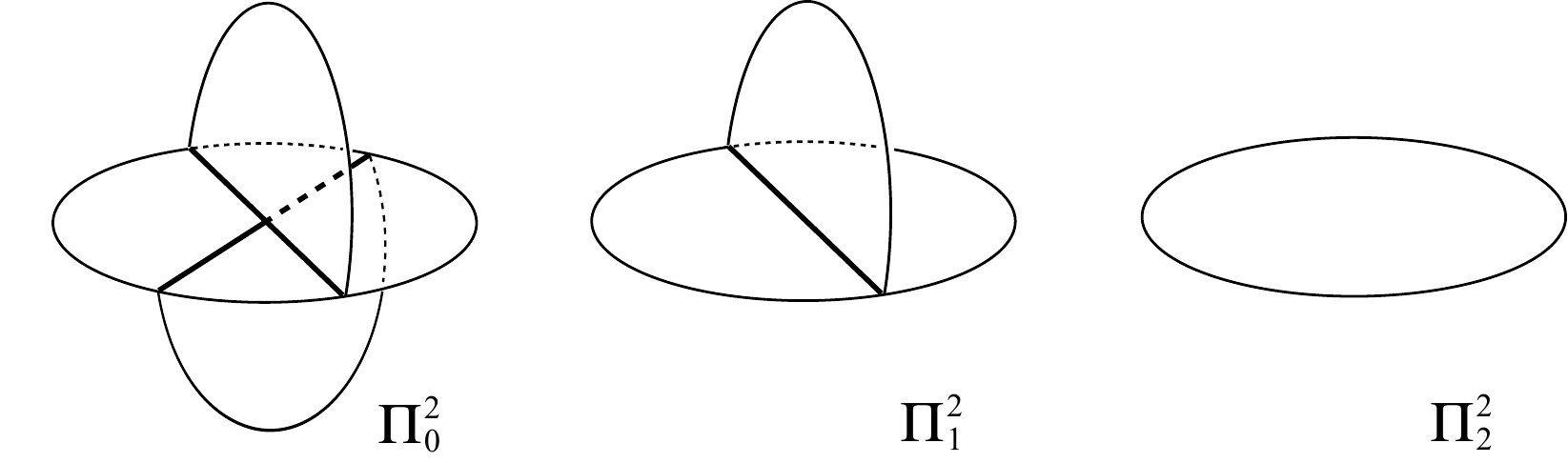}
 \end{center}
 \caption{The local models of a simple polyhedron of dimension 2.}
 \label{models:fig}
\end{figure}

The polyhedron $\Pi^n$ has a natural triangulation induced by that of $\Delta$. 

\begin{prop}
A point $x\in\interior{\Pi^n}$ has intrinsic dimension $k$ if and only
if it is of type $k$. 
\end{prop}
\begin{proof}
Since $\partial\Pi^{n-k}$ is not a suspension, a point of type $k$ has intrinsic dimension $k$ by Proposition \ref{suspension:prop}. 
\end{proof}

The polyhedron $\Pi^n$ may be constructed recursively. In the
following, we see both $\Pi^{n-1}$ and $S^{n-1}$ inside $D^n$. See Fig. \ref{link:fig}.

\begin{figure}
 \begin{center}
  \includegraphics[width = 6cm]{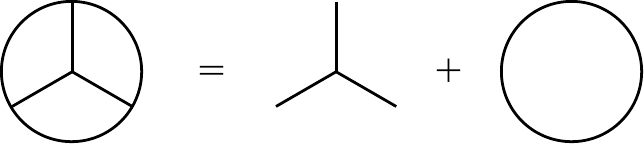}
 \end{center}
 \caption{We have $\partial\Pi^n \cong \Pi^{n-1}\cup S^{n-1}$. Here $n=2$.}
 \label{link:fig}
\end{figure}

\begin{prop}\label{recursive:prop}
We have $\partial \Pi^n \cong \Pi^{n-1}\cup  S^{n-1}$.
\end{prop}
\begin{proof}
Take a vertex $v$ and the opposite face $f$ in $\Delta^{n+1}$.
The $(n-1)$-skeleton of $\Delta^{n+1}$ is the union of $\partial f$ and a cone over the $(n-2)$-skeleton of $f$ with base $v$.
\end{proof}

\begin{cor}\label{recursive:cor}
We have $\Pi^n\cong \big(\Pi^{n-1}\times [0,1]\big) \cup
\big(D^n\times \{0\}\big)$.
\end{cor}

\begin{figure}
 \begin{center}
  \includegraphics[width = 3.5cm]{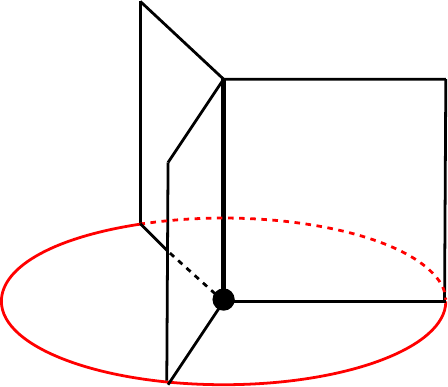}
 \end{center}
 \caption{We have $\Pi^n \cong\big(\Pi^{n-1}\times [0,1]\big) \cup \big(D^n\times \{0\}\big)$. Here $n=2$: the disc $D^2\times \{0\}$ is horizontal (in red) and $\Pi^1\times [0,1]$ is vertical (in black).}
 \label{star:fig}
\end{figure}

See Fig \ref{star:fig}. The following is an easy corollary of Proposition \ref{suspension:prop}.
\begin{ex} \label{suspension:ex}
If $X$ is a polyhedron such that $X\times [-1,1]^h\cong \Pi^n_k$ then
$X\cong \Pi^n_{k-h}$.
\end{ex}

\subsection{Simple polyhedron} \label{simple:subsection}
\begin{figure}
 \begin{center}
 \resizebox{12 cm}{!}{\input{simple.pdftex_t}}
  \end{center}  
 \caption{Simple polyhedra of dimension $n=1$ (a circle and a trivalent graph) and $n=2$ (a sphere and a torus with two discs attached).}
 \label{simple:fig}
 \end{figure}
\begin{defn} \label{simple:defn} A compact polyhedron $P^n$ is \emph{simple} if every point of $P$ is of some type $k$ (that is, its link is homeomorphic to $\partial\Pi^n_k$).
\end{defn}
See some examples in Fig.~\ref{simple:fig}. 
A point of type $0$ is called a \emph{vertex}.
In this paper, every simple polyhedron $P\subset \interior{M^n}$
contained in some manifold 
$M^n$ will be tacitly assumed to have codimension $1$ and to be locally
unknotted, see Section \ref{intrinsic:subsection}. This is
equivalent to require that $P$ is \emph{properly embedded} in
Matveev's sense \cite{Mat:special}: the equivalence is proved in 
Section \ref{unknotted:subsection}. Local unknottedness is actually automatic in dimension $n\leqslant 4$, see Remark \ref{automatic:rem} below.

\begin{ex} \label{simple:ex}
The polyhedron $\partial \Pi^n$ is simple with $n+2$ vertices.
\end{ex}
The exercise is also proved as Corollary \ref{simple:cor} below.

\subsection{Complexity} \label{complexity:subsection}

\begin{figure}
 \begin{center}
 \includegraphics[width = 12cm]{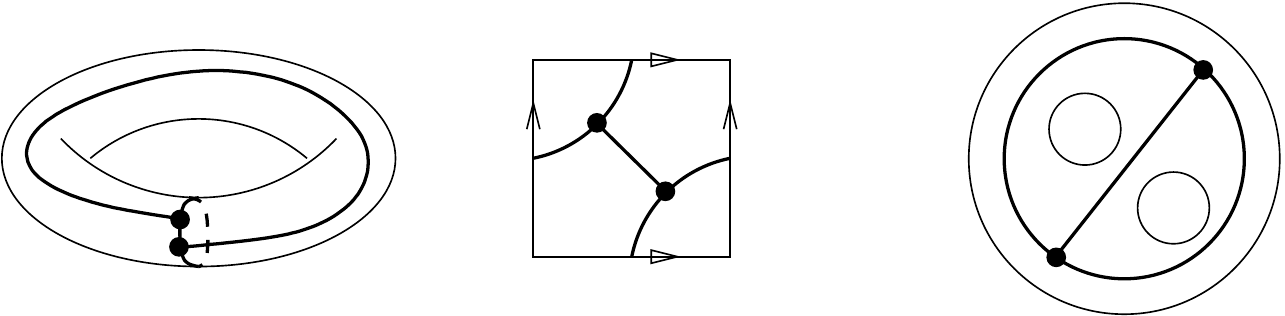}
  \end{center} 
   \caption{A spine of the torus (left and center) and of the
   pair-of-pants (right).}
 \label{spines:fig}
\end{figure}

A spine of a manifold is usually defined as a subpolyhedron onto which
the manifold collapses. This definition however applies only to
manifolds with boundary: in order to extend it to closed manifolds, we
allow the removal of an arbitrary number of open balls.
\begin{defn}
Let $M$ be a compact manifold. A subpolyhedron $P\subset \interior{M}$ is  a \emph{spine} of $M$ if there are some disjoint discs $D_1,\ldots, D_k\subset \interior M$, disjoint also from $P$, such that $M\setminus \interior{D_1\cup\ldots\cup D_k}$ collapses onto $P$.
\end{defn}

See some examples in Fig.~\ref{spines:fig}. We are now ready to define the complexity of a manifold.
\begin{defn}
The \emph{complexity} $c(M)$ of a compact manifold $M$ is the minimum number of vertices in a simple spine of $M$. 
\end{defn}

Every compact manifold admits a simple spine (see \cite{Mat:special} or Corollary \ref{exists:cor} below) and hence this quantity is indeed finite. A simple spine $P\subset M$ is \emph{minimal} if it has $c(M)$ vertices.

\subsection{Examples} \label{examples:subsection}
The equator $(n-1)$-sphere is a simple spine of $S^n$: the $n$-sphere
collapses to it after removing two small balls centered at the
poles. Analogously, a hyperplane is a simple spine of $\matRP^n$
without vertices ($\matRP^n$ collapses to it after removing one
ball). When $n\geqslant 2$ these spines have no vertices and 
therefore $c(S^n)=c(\matRP^n)=0$.  When $n=1$, the circle $S^1$ has a point as a simple spine, which is indeed a vertex, and hence $c(S^1)=1$.

Figure \ref{spines:fig} shows a spine with $2$ vertices of the $2$-torus $T$: hence $c(T)\leqslant 2$. It is easy to see that $T$ has no spine with lower number of vertices, and hence $c(T)=2$. A similar argument shows the following. 

\begin{ex} \label{surface:ex}
The complexity $c(\Sigma)$ of a closed surface $\Sigma$ is 
\begin{itemize}
\item $c(\Sigma) = \max\{2-2\chi(\Sigma),0\}$ if $\Sigma$ is closed,
\item $c(\Sigma) = \max\{-2\chi(\Sigma),0\}$ if  $\Sigma$ has boundary.
\end{itemize}
\end{ex}

The surfaces having complexity zero are $S^2$, $\matRP^2$, the annulus, and the M\"obius strip. They all have a circle as a spine without vertices.

Many examples in dimension 3 can be found in the literature \cite{Bu, MaPe, survey, Mat, Mat:book}, so we turn to higher-dimensional manifolds. A nice spine for $\matCP^n$ can be described by using a technique which was inspired to us by tropical geometry \cite{Mik}. Consider the projection
$$
\begin{array}{rlcl}
p: & \matCP^n & \longrightarrow & \Delta^n \\
 & \left[z_0,\ldots,z_{n}\right] & \longmapsto & \left(\frac{|z_0|}{|z_0|+\ldots +|z_{n}|},\ldots, \frac {|z_{n}|}{|z_0|+\ldots+|z_{n}|}\right).
 \end{array}
$$

Consider $\Pi^{n-1}$ dually embedded in $\Delta^n$.
The counterimage $p^{-1}(\Pi^{n-1})$ is a simple spine of $\matCP^n$ without vertices (its complement consists of $n+1$ open balls ``centered'' at the points $[0,\ldots,0,1,0,\ldots,0]$). Therefore $c(\matCP^n)=0$.

The spine of $\matCP^2$ fibers over $\Pi^1$. It consists of three solid tori attached to one 2-torus. We find such a spine also from a different construction. Let $M^4$ be a closed 4-manifold which decomposes with 0-, 2-, and 4-handles only. The attaching of the 2-handles is encoded by a framed link $L\subset S^3$. Let $P$ be the union of the boundaries of all the handles involved. It consists of a 3-sphere $S^3$ plus one solid torus attached to (a regular neighborhood of) each component of $L$. When $L$ is the 1-framed unknot we find $M^4=\matCP^2$ and we get the same spine as above. In general, we get a simple spine of $M$ without vertices (all points are of type $3$ or $2$). Therefore $c(M^4)=0$.

\section{Collars} \label{collars:section}
As proved by Matveev \cite{Mat:special}, a locally unknotted simple polyhedron $P\subset M$ has a kind of collar, similar to a collar of the boundary of a manifold. We introduce the collar by defining the \emph{cut map} in Section \ref{cut:subsection}. To do this, we first need to prove that Matveev's notion of local flatness (which is more useful in the context of simple spines) coincides with the general one introduced in Section \ref{intrinsic:subsection}.

\subsection{Matveev's definition} \label{unknotted:subsection}
Matveev introduced in \cite{Mat:special} a different definition of local unknottedness for simple polyhedra, which is more useful here. We show that it coincides with the general one introduced in Section \ref{intrinsic:subsection}. The proof is not strictly necessary (we could use Matveev's notion and forget about the general one), but we include it for completeness.

We defined the pair $(D^n,\Pi^{n-1}_k)$ in Section \ref{local:subsection}. The following definition is due to Matveev \cite{Mat:special}.

\begin{defn} A simple polyhedron $P^{n-1}\subset M^n$ in a manifold $M^n$ is \emph{properly embedded} if the link of every point in $(M,P)$ is homeomorphic to $(S^{n-1}, \partial\Pi^{n-1}_k)$ for some $k$. 
\end{defn}

\begin{prop} \label{equivalent:prop}
A simple polyhedron $P\subset \interior M$ of codimension 1 is locally unknotted if and only if it is properly embedded.
\end{prop}
\begin{proof}
It is easy to see that a properly embedded $P^{n-1}\subset M^n$ is locally unknotted. We prove the converse by induction on $n$. The case $n=1$ is trivial, so we assume $n\geqslant 2$. Let $x$ be a point of $P$, of some type $k$. The link of $x$ in $(M,P)$ is homeomorphic to $(S^{n-1},Y^{n-2})$ with $Y^{n-2}\cong\partial \Pi^{n-1}_k$. We must show that the homeomorphism extend to pairs, \emph{i.e.}~that $(S^{n-1},Y^{n-2})\cong(S^{n-1}, \partial \Pi^{n-1}_k)$.

Since $P$ is locally unknotted, the pair $(S^{n-1}, Y^{n-2})$ is also locally unknotted by Exercise \ref{unknotted:ex}. The polyhedron $Y$ is simple by Exercise \ref{simple:ex}, and is hence properly embedded by our induction hypothesis. Since $Y\cong\partial \Pi^{n-1}_k$ is a \emph{special} polyhedron (\emph{i.e.}~a simple polyhedron whose $(n-2)$-components are discs), \cite[Theorem 3]{Mat:special} ensures that the homeomorphism $Y\cong \partial\Pi^{n-1}_k$ indeed extends to a 
regular neighborhood and hence to the whole of $S^{n-1}$, as required. 
\end{proof}

\begin{rem} \label{automatic:rem}
Local unknottedness is automatic in dimension $n\leqslant 4$: in these dimensions, every embedding of $\partial\Pi^{n-1}_k$ in $S^{n-1}$ is in fact easily seen to be standard. In dimension $5$, a nonstandard pair $(S^4,S^3)$, if it exists, could lead to nonstandard embeddings of $\Pi^3_k$ in $S^4$. \end{rem}

\subsection{Cut map}\label{cut:subsection}
\begin{figure}
 \begin{center}
 \includegraphics[width = 12.5cm]{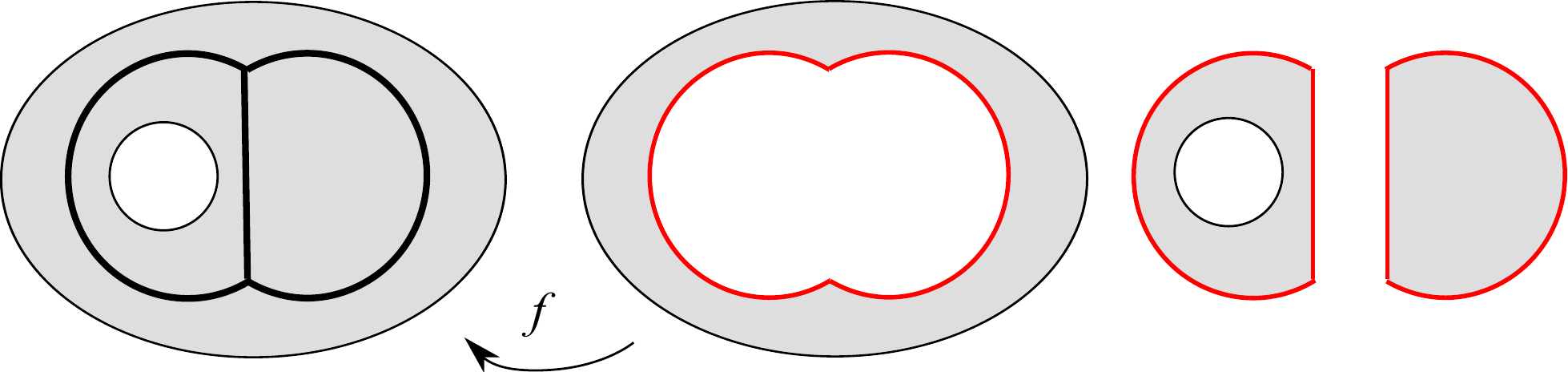}
 \end{center}
\caption{The cut map $f:M_P\to M$ cuts the manifold $M$ along the simple polyhedron $P$. The red boundary is $\partial_0 M_P$. }
\label{cut:fig}
\end{figure}

As noted by Matveev \cite{Mat:special}, the locally unknotted embedding of a simple polyhedron allows us to define a \emph{collar}, similar to the collar of a boundary in a manifold.

Let $P\subset \interior M$ be a simple polyhedron in a compact manifold. By cutting $M$ along $P$ as suggested in Fig.~\ref{cut:fig}, we get a manifold $M_P$ with boundary and a surjective map $f:M_P\to M$. 

The set $f^{-1}(P)\subset M_P$ consists of some components of $\partial M_P$, which we denote by $\partial_0 M_P$. The map $f$ is a local embedding. It is $(n-k+1)$-to-$1$ over a point of type $k$ in $(M,P)$. In particular, it restricts to a homeomorphism of $M_P\setminus\partial_0 M_P$ onto $M\setminus P$.

Regular neighborhoods $R(P)$ of $P$ in $M$ correspond via $f$ to collars of $\partial_0 M_P$. The function $f$, restricted to one such collar, gives a \emph{collar} $\partial R(P)\times [0,1]\to R(P)$ of $P$, as shown in \cite{Mat:special}. This discussion implies in particular the following.
\begin{prop} \label{cut:prop}
Let $P\subset \interior M$ be a simple polyhedron. It is a spine of $M$ if and only if 
$$M_P = N\times [0,1] \sqcup D_1\sqcup\ldots \sqcup D_k $$
where $D_1,\ldots,D_k$ are discs, $N$ is a possibly disconnected $(n-1)$-manifold, and $\partial_0 M_P = N\times 0$.
\end{prop}
In other words, a simple polyhedron $P\subset\interior M$ is a spine if and only if $M_P$ consists of a collar and some discs.

\section{Triangulations} \label{triangulations:section}
We describe here a construction which builds a simple polyhedron $P\subset M$ from a triangulation of $M$ and a partition of its vertices. From this we will deduce that $c(M)\leqslant t(M)$ for any compact $M$.

\subsection{Dual models} \label{models:subsection}
We generalize the dual description of $\Pi^n$ inside the simplex $\Delta=\Delta^{n+1}$ to the polyhedra $\Pi^n_k$. Let $\calP=\{V_0,\ldots, V_k\}$ be a partition of the set $V$ of vertices of $\Delta$. Every $V_i$ spans a face $f_i$ of $\Delta$. 

\begin{defn} \label{dual:defn}
The polyhedron \emph{dual} to $(\Delta,\calP)$ is
$$ Z = \bigcup_{i=0}^k \lk(f_i,\Delta'). $$
\end{defn}
When $\calP=\{V\}$ we have $Z=\emptyset$. For the other cases, we have the following.
\begin{prop} \label{Z:prop}
We have $Z\cong \Pi^n_{n+1-k}$.
\end{prop}
\begin{proof}
If each $f_i$ consists of one vertex we get the dual description of $\Pi^n$ and we are done.
Otherwise, let $\Delta^k$ be a $k$-dimensional simplex, with vertices $w_0,\ldots, w_k$. By sending $V_i$ to $w_i$ we get a simplicial map $\phi:\Delta^{n+1}\to\Delta^k$. This induces another simplicial map $\phi':(\Delta^{n+1})'\to(\Delta^k)'$ between the derived complexes. Let $Z_*$ be the polyhedron dual to $(\Delta^k, \{w_0,\ldots,w_k\})$. We have $f^{-1}(Z_*)=Z$ and $Z_*\cong \Pi^k$.

On a small neighborhood of the center $C$ of $(\Delta^{n+1})'$ the map $f'$ is isomorphic to the projection $D^k\times D^{n-k+1}\to D^k$. Therefore the star of $C$ in $Z$ is homeomorphic to $\Pi^{k-1}\times D^{n-k+1}\cong\Pi^n_{n-k+1}$. Such a star is homeomorphic to $Z$ and we are done.
\end{proof}

See a couple of examples in Fig.~\ref{dual_models:fig}.

\begin{figure}
 \begin{center}
 \includegraphics[width = 10cm]{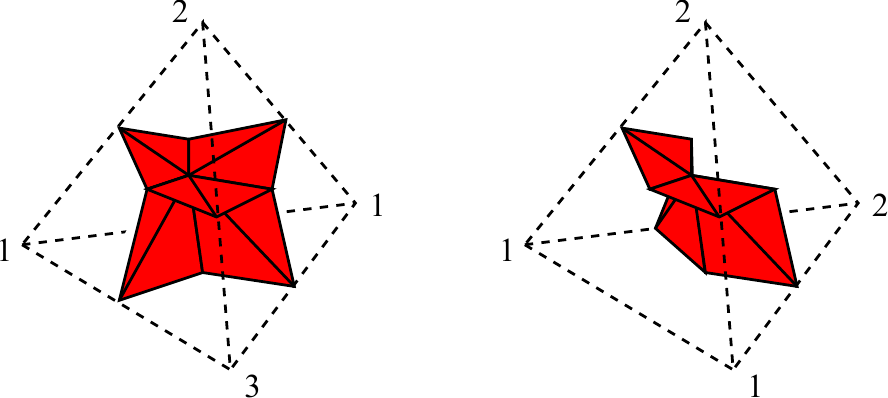}
 \end{center}
\caption{Models dual to some partitions of the vertices. Here we find $\Pi_1^2$ and $\Pi_2^2$.}
\label{dual_models:fig}
\end{figure}

\subsection{Simple polyhedra dual to triangulations} \label{dual:subsection}

\begin{figure}
 \begin{center}
 \includegraphics[width = 12.5cm]{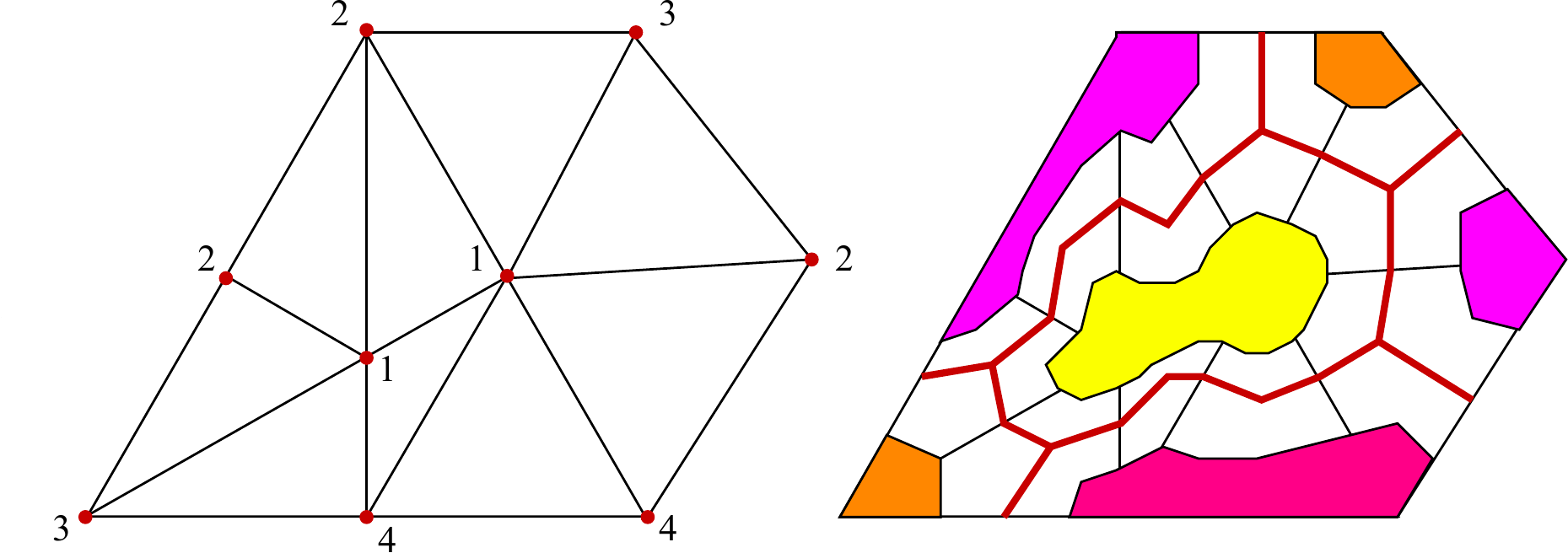}
 \end{center}
\caption{A triangulation $T$ and a partition of the vertices induce a simple polyhedron (in red) realized as a subcomplex of $T'$ and some submanifolds $M_V$ (the coloured regions) realized as subcomplexes of $T''$.}
\label{partition:fig}
\end{figure}

Let $M$ be a compact manifold and $T$ a triangulation of $M$. Let $\calP$ be a partition of the vertices of $T$. In each simplex $\sigma$ of $T$ we have an induced partition of its vertices and hence a dual polyhedron $P_\sigma\subset\sigma$. 

\begin{defn} \label{dual:global:defn}
The polyhedron \emph{dual} to $(T,\calP)$ is 
$$P = \bigcup_{\sigma \in T} P_\sigma$$
\end{defn}

The dual polyhedron is a subcomplex of the barycentric subdivision $T'$. We say that $\calP$ \emph{respects $\partial M$} if for every connected component $X$ of $\partial M$ all the vertices in $X$ belong to the same set of $\calP$.

For every set $V\in \calP$ of the partition, we define the submanifold $M_V\subset M$ as the regular neighborhood in $T''$ of the union of all simplexes in $T$ whose vertices lie in $V$. See Fig.~\ref{partition:fig}. We have the following.

\begin{prop} \label{dual:simple:prop}
Let $\calP$ be a partition that respects $\partial M$ and $P$ be the polyhedron dual to $(T,\calP)$. Then $P$ is simple and the following holds.
\begin{enumerate}
\item The vertices of $P$ are the barycenters of the simplexes in $T$ whose vertices lie in $n+1$ distinct sets of $\calP$.
\item The regular neighborhood of $P$ in $T''$ is $M\setminus\interior {\cup_{V\in\calP} M_V}.$
\end{enumerate}
\end{prop}
\begin{proof}
The proof is straightforward. Note that by our assumption on $\calP$ every boundary component lies in some $M_V$ and hence $P$ does not intersect $\partial M$.
\end{proof}
Let $t(M)$ be the minimum number of simplexes arising in a triangulation of $M$.

\begin{cor} \label{exists:cor}
We have $c(M)\leqslant t(M)$ for every compact $M$.
\end{cor}
\begin{proof}
Let $T$ be a triangulation with $t=t(M)$ simplexes. Let $\calP$ be the following partition: two vertices $v,v'$ belong to the same set if and only if $v=v'$ or $v$, $v'$ belong to the same boundary component of $M$. The polyhedron dual to $(T,\calP)$ is a spine of $M$ by Proposition \ref{dual:simple:prop}-(2), since $M_V$ consists of discs (stars of the inner vertices) and a collar of the boundary. It has at most $t$ vertices by Proposition \ref{dual:simple:prop}-(1).
\end{proof}

\begin{rem} \label{singular:rem}
The term ``triangulation'' is sometimes used for short in dimensions 2 and 3 to indicate a \emph{singular triangulation}, \emph{i.e.}~the realization of a manifold $M^n$ as the union of some $n$-simplexes whose faces are identified in pairs via some simplicial maps. This is not the case here: in this paper we employ the word ``triangulation'' only in its original PL meaning.
\end{rem}

\subsection{Simple subpolyhedra of $\partial\Pi^n$.}
Proposition \ref{dual:simple:prop} yields the following.
\begin{cor} \label{simple:cor}
The polyhedron $\partial\Pi^n_k$ is simple.
\end{cor}
\begin{proof}
Represent $\Pi^n_k$ as the dual of some partition $(\Delta,\calP)$. The boundary is the simple polyhedron dual to the same partition $(\partial\Delta,\calP)$.
\end{proof}

Proposition \ref{Z:prop} shows that $\partial \Pi^n_k$ is homeomorphic to some simple subpolyhedron of $\partial \Pi^n$. Conversely, we have the following result (which will be used in Section \ref{alternative:section}).
\begin{prop} \label{subpolyhedra:prop}
See $\partial \Pi^n$ dually contained in $\partial\Delta$.
If $n\geqslant 2$, every simple subpolyhedron of $\partial \Pi^n$ is dual to some partition $\calP$ of the vertices of $\Delta$ and is hence homeomorphic to $\partial \Pi^n_k$ for some $k$.
\end{prop}
\begin{proof}
We prove this result by induction on $n$. The case $n=2$ is easy, so we turn to the case $n\geqslant 3$. Let $X$ be a simple subpolyhedron of $\partial\Pi^n$. Let $f_0,\ldots,f_{n+1}$ be the facets of $\Delta$. A simple subpolyhedron of a simple polyhedron of the same dimension $n-1$ is necessarily the closure of the union of some $(n-1)$-components. Therefore each $f_i\cap X$ is a cone over $\partial f_i\cap X$, which is a simple subpolyhedron of $\partial f_i\cap\partial \Pi^n\cong\partial\Pi^{n-1}$. By induction, $\partial f_i\cap X$ is dual to some partition $\calP_i$ of the vertices of $f_i$.

Now define a partition $\calP$ of the vertices of $\Delta$ as follows: two vertices belong to the same set if they do in some $\calP_i$. Note that if they do belong to the same set in some $\calP_i$ then they do so in any $\calP_j$, and that the transitivity property of this equivalence relation holds because $n\geqslant 3$. The polyhedron $X$ is then dual to $(\Delta,\calP)$.
\end{proof}

\section{Drilling} \label{drilling:section}
Generic soap bubbles in $\matR^3$ form a simple polyhedron. Moreover, if a new bubble appears generically somewhere, the polyhedron remains simple. This fact can be generalized to any dimension, as follows.

Let $Q\subset M$ be a simple polyhedron in a manifold. Let $K\subset \interior M$ be any compact subpolyhedron. The operation of \emph{drilling $Q$ along $K$} consists of removing from $Q$ a small regular neighborhood of $K$ and adding its boundary as in Figg.~\ref{bubbling_1:fig} and~\ref{bubbling:fig}.
More precisely, let $T$ be a triangulation of $(M,K,Q)$. Let $R=R(K,T'')$ be the regular neighborhood of $K$ in the twice subdivided $T''$. The result of this operation is the polyhedron
$$P=(Q\setminus R)\cup\partial R.$$

\begin{figure}
 \begin{center}
 \includegraphics[width = 12cm]{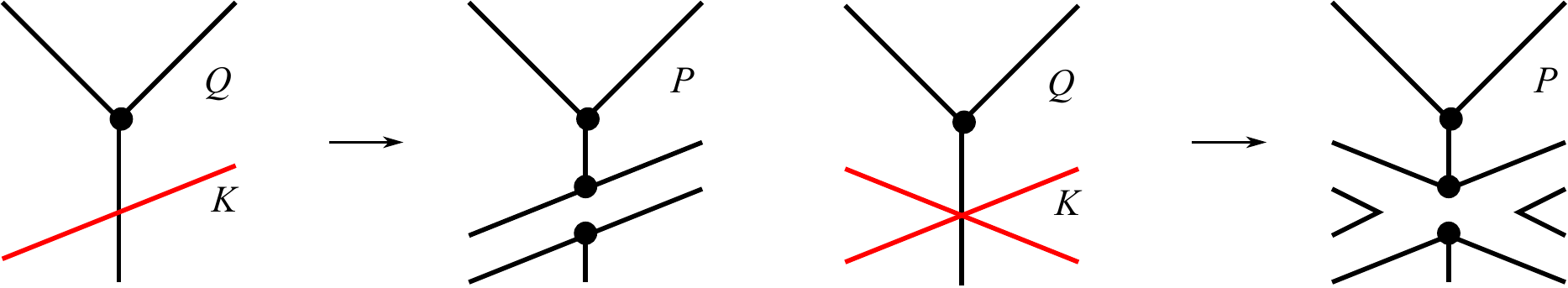}
 \end{center}
\caption{The simple polyhedron $P$ is obtained from $Q$ by drilling along $K$.}
\label{bubbling_1:fig}
\end{figure}

\begin{figure}
 \begin{center}
 \resizebox{11cm}{!}{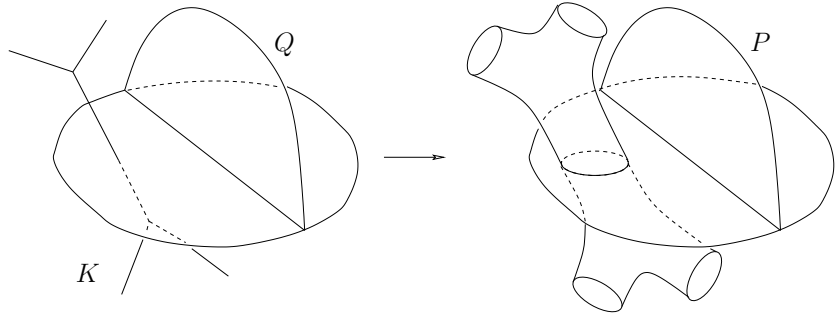}
 \end{center}
\caption{The simple polyhedron $P$ is obtained from $Q$ by drilling along $K$.}
\label{bubbling:fig}
\end{figure}

\begin{lemma}\label{bubble:lemma}
The polyhedron $P$ is simple. If $K$ does not intersect the $1$-skeleton of $Q$, then $P$ has the same vertices as $Q$.
\end{lemma}
\begin{proof}
We have $P = (Q\setminus R)\cup\partial R$. We have to check that every point $x\in P$ is of some type $k$. If $x\in Q\setminus \partial R$ we are done because $Q$ is simple. Suppose $x\in \partial R$. Let $k\geqslant 1$ be the type of $x$ in $(M,Q)$. We show that $x$ is of type $k-1$ in $(M,P)$: in particular, $P$ is simple.

\begin{figure}
 \begin{center}
 \includegraphics[width = 9cm]{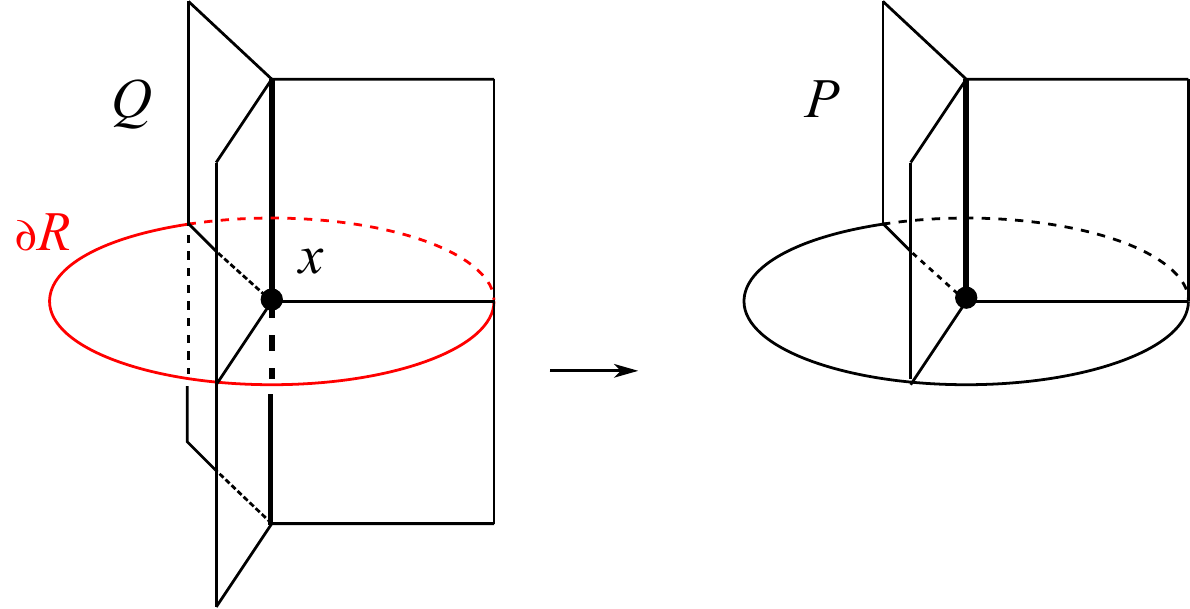}
 \end{center}
\caption{At $x$ the polyhedra $Q$ and $\partial R$ intersect transversely, so $\partial R$ cuts $Q$ into two halves. Then $P$ is obtained by discarding the half of $Q$ lying inside $R$ and adding $\partial R$.}
\label{bubbling_proof:fig}
\end{figure}

If $k=n$, \emph{i.e.}~$x\not\in Q$, then $x$ is of type $(n-1)$ in $P$ because $\partial R$ is a $(n-1)$-manifold. If $k<n$, the polyhedra $\partial R$ and $Q$ intersect transversely at $x$ (in the sense of Armstrong, see \cite{Arm}). See an example in Fig.~\ref{bubbling_proof:fig}: locally, $\partial R$ is a horizontal disc $D^{n-1}\times 0$ and $Q$ is a vertical product $Y\times [-1,1]$. Exercise \ref{suspension:ex} implies that $Y\cong\Pi^{n-1}_{k-1}$.

The star of $x$ in $P$ is thus homeomorphic to 
$$\Big(\Pi^{n-1}_{k-1}\times [0,1]\Big) \cup \Big(D^n\times 0\Big) \cong D^{k-1}\times\bigg(\Big(\Pi^{n-k}\times [0,1]\Big)\cup \Big(D^{n-k+1}\times 0\Big)\bigg).$$
By Corollary \ref{recursive:cor}, this is homeomoprhic to $D^{k-1}\times \Pi^{n-k+1} \cong \Pi^n_{k-1}$. Therefore $x$ is of type $k-1$ in $P$.

Finally note that if $K$ does not intersect the $1$-skeleton of $Q$ neither does $R$. Therefore every point in $Q\cap\partial R$ is of type $k>1$ in $Q$, and hence of type $k-1>0$ in $P$: no new vertices are added to $Q$.  
\end{proof}

\section{Alternative definitions} \label{alternative:section}
Matveev's original definition of complexity $c(M^3)$ for a $3$-manifold $M^3$ was slightly different from ours. We prove here that the two definitions coincide. 

A couple of natural variations might be done in our definition of complexity. The definition of ``simple polyhedron'' can be weakened by allowing the presence of low-dimensional material. This choice is natural, since it allows to consider a point as a spine of $D^n$ or $S^n$. Matveev called such polyhedra \emph{almost simple}. On the other hand, the definition of ``spine'' can be strengthened, by allowing the removal of one ball only when strictly necessary, \emph{i.e.}~when the manifold is closed. We call this more restricted notion a \emph{strict spine}.

We therefore get $2\times 2=4$ possible definitions of the complexity of a manifold. Luckily, it turns out that three of them coincide. These include ours and Matveev's definition (in dimension 3).

\subsection{Almost simple polyhedra}
Matveev employed in dimension 3 a more relaxed notion of polyhedron, called \emph{almost simple} \cite{Mat}. We propose the following generalization to all dimensions.

\begin{defn} \label{almost:simple:defn} Let $M^n$ be a compact manifold.
A compact subpolyhedron $P\subset \interior M$ is \emph{almost simple} if the link of every point in $(M,P)$ is homeomorphic to $(S^{n-1}, L)$ for some subpolyhedron $L\subset \partial\Pi^{n-1}\subset S^{n-1}$. 
\end{defn}

\begin{figure}
 \begin{center}
 \includegraphics[width = 10cm]{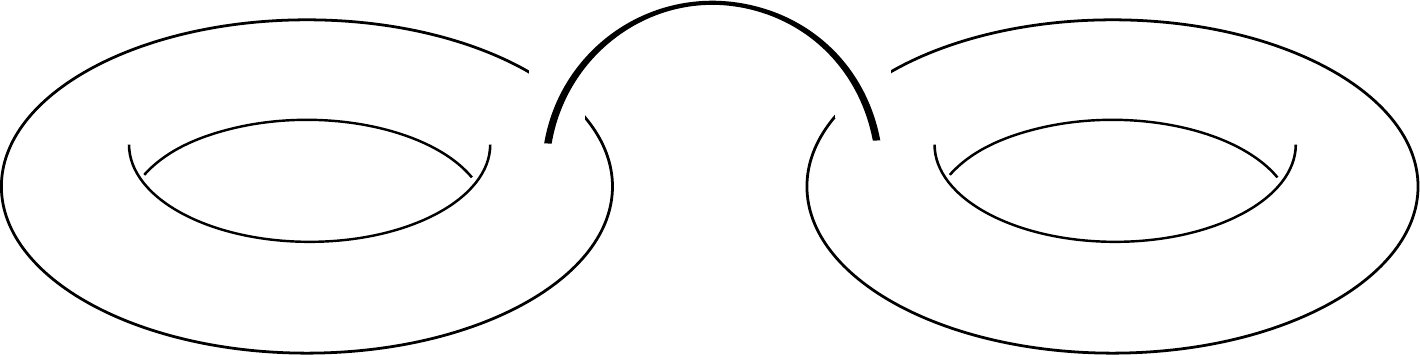}
 \end{center}
\caption{An almost simple spine of the boundary-connected sum $(T^2\times D^1)\# (T^2\times D^1)$. It consists of two tori joined by an arc. It has no vertices.}
\label{almost_simple:fig}
\end{figure}

See an example in Fig~\ref{almost_simple:fig}. A \emph{vertex} of $P$ is a point whose link is homeomorphic to $(S^{n-1},\partial\Pi^{n-1})$. We now define $c^{\rm alm}(M)$ as the minimum number of vertices of an almost simple spine of $M$. 

\begin{example} A point $\{pt\}\subset S^n$ is an almost simple spine of the $n$-sphere. Note that a point is \emph{not} a vertex when $n\geqslant 2$, and hence $c^{\rm alm}(S^n)=0$ for all $n\geqslant 2$. A hyperplane $H\subset \matCP^n$ is an almost simple spine of complex projective space without vertices, and hence $c^{\rm alm}(\matCP^n)=0$. This spine is not simple (it has codimension 2): note that the construction of a simple spine of $\matCP^n$ without vertices is less immediate, see Section \ref{examples:subsection}.
\end{example}

We show below that $c=c^{\rm alm}$: to prove this, we need a couple of preliminary lemmas, which show how to construct a simple spine from an almost simple one without increasing the number of vertices. This is done first by collapsing (Lemma \ref{collapse:lemma}) and then by drilling along a triangulation of the low-dimensional part $K$ of the spine (Lemma \ref{low:attach:lemma}).

\begin{figure}
 \begin{center}
 \includegraphics[width = 12.5cm]{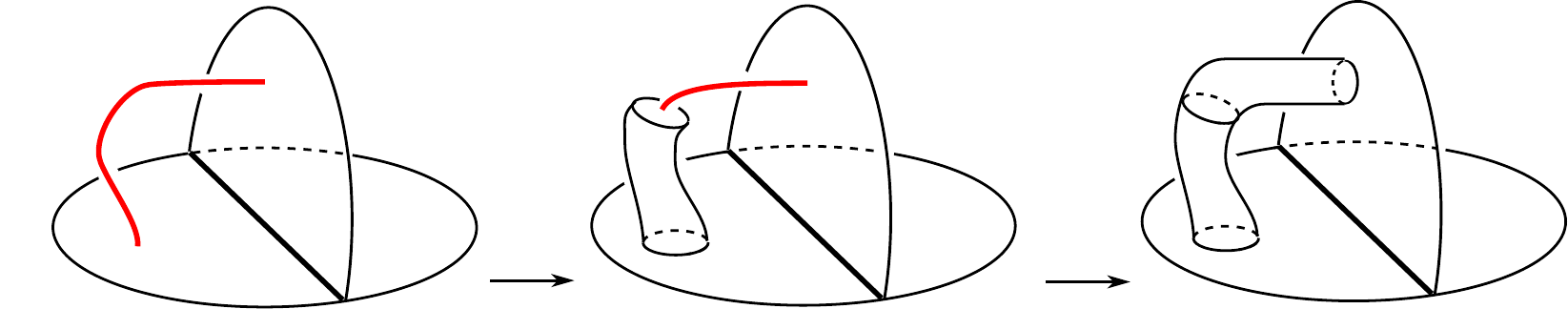}
 \end{center}
\caption{Let $Q\cup K$ be a spine of $M$, such that $Q$ is simple and $\dim K<\dim Q$. We obtain a simple spine of $M$ by subdividing $K$ into cells and then drilling them inductively. Each drilling produces a new ball in the complement. Here, $K$ is an arc made of two 1-cells.}
\label{low_dim:fig}
\end{figure}

\begin{lemma} \label{collapse:lemma}
Every almost simple polyhedron $P\subset\interior {M^n}$ collapses onto $Q\cup K$ where $Q$ is simple and $\dim K<n-1$. Every vertex of $Q$ is also a vertex of $P$.
\end{lemma}
\begin{proof}
We prove this by induction on $n=\dim M$. If $n=2$ it is easy, so we turn to the case $n\geqslant 3$. Take a triangulation of $(M,P)$ and collapse $P$ as more as possible. The resulting polyhedron is some $Q\cup K\subset P$, where $Q$ (resp.~$K$) is the closure of the set of all points whose link has dimension $n-2$ (resp.~$<n-2$). 

We now prove that $Q$ is simple. Since $P$ is almost simple, the link of $x$ in $(M,Q,K)$ is homeomorphic to $(S^{n-1},Q',K')$ for some $Q'\cup K'\subset \partial\Pi^n\subset S^{n-1}$. The polyhedron $Q'\cup K'$ cannot be collapsed onto a proper subpolyhedron, because $(Q,K)$ cannot. By our induction hypothesis $Q'$ is hence simple. Proposition \ref{subpolyhedra:prop} implies that $Q'\cong\partial\Pi^{n-2}_k$ and hence $x$ is of type $k$ in $Q$.
\end{proof}

\begin{lemma}\label{low:attach:lemma}
If $M^n$ has a spine $Q\cup K\subset M^n$ such that $Q$ is simple with $t$ vertices and $K$ has dimension $<n-1$, then $c(M^n)\leqslant t$.
\end{lemma}
\begin{proof}
Take a triangulation $T$ of $(M,Q,K)$. Let $\sigma_1,\ldots,\sigma_k$ be the simplexes of $K$ that are not contained in $Q$, and not contained in any higher-dimensional simplex of $K$. We have $Q\cup K = Q\cup\sigma_1\cup\ldots\cup \sigma_k$.

Each $\sigma_i$ is a cell (\emph{i.e.}~a disc) of dimension $<n-1$. We want to drill inductively along each $\sigma_i$. Since each $\sigma_i$ is a cell, each drilling produces a new open ball in the complement, so the final simple polyhedron is a spine. In order to not create new vertices, at each step we put the cells in general position. See Fig.~\ref{low_dim:fig}.

More precisely, we construct for each $i=0,\ldots,k$ a simple polyhedron $Q_i$ with $t$ vertices and some cells $\sigma_1^i,\ldots,\sigma_{k-i}^i$ of dimension $<n-1$ intersecting themselves and $Q_i$ only at their boundaries, such that $Q_i\cup\sigma_1^i\cup\ldots\cup\sigma_{k-i}^i$ is a spine of $M^n$. For $i=0$, take $Q_0=Q$ and $\sigma_j^0 = \sigma_j$. 

Let now $Q_i$ and $\sigma_1^i,\ldots,\sigma_{k-i}^i$ be defined for some $i<k$. Since $\dim \sigma_j^i<n-1$, we can perturb the cells so that they do not intersect the $1$-skeleton of $Q$. To do this, we use a collar of $P$ (see Section \ref{cut:subsection}): we lift the cells to $M_{Q_i}$, perturb them slightly, and project them back. The perturbed polyhedron is still a spine.

Let $Q_{i+1}$ be obtained from $Q_i$ by drilling along $\sigma_{k-i}^i$. The polyhedron $Q_{i+1}$ is simple with $t$ vertices by Lemma \ref{bubble:lemma}. To drill we use a triangulation $T$ of the whole data $(M,Q_i,\sigma_1^i,\ldots,\sigma_{k-i}^i)$: this ensures that $\sigma_j^{i+1} = \overline{\sigma_j^i\setminus Q_{i+1}}$ is still a cell.

Finally, $Q_k$ is a simple spine of $M$ with $t$ vertices.
\end{proof}

\begin{teo} \label{almost:simple:teo}
We have $c^{\rm alm}(M)=c(M)$ for every compact manifold $M$.
\end{teo}
\begin{proof}
A simple spine is almost simple, hence $c^{\rm alm}(M)\leqslant c(M)$. Conversely, an almost simple spine $P$ of $M$ with $c^{\rm alm}(M)$ vertices collapses by Lemma \ref{collapse:lemma} onto a $Q\cup K$ such that $Q$ has at most $c^{\rm alm}(M)$ vertices, and hence $c(M)\leqslant c^{\rm alm}(M)$ by Lemma \ref{low:attach:lemma}.
\end{proof}

\subsection{Strict spines} \label{strict:subsection}
Matveev's notion of ``spine'' was more rigid than ours. 
\begin{defn}
Let a \emph{strict spine} $P\subset \interior M$ of a compact manifold $M$ with boundary be a compact $P$ onto which $M$ collapses. A \emph{strict spine} of a closed $M$ is defined as a strict spine of $M\setminus \interior D$ for some disc $D$.
\end{defn}
In contrast with spines, the complement of a strict spine contains one ball only if strictly necessary. Let $c^{\rm alm}_{\rm str}(M)$ be the minimum number of vertices of an almost simple strict spine. Matveev's original definition of the complexity of a 3-manifold $M^3$ is precisely $c^{\rm alm}_{\rm str}(M^3)$. We can finally show that his definition coincides with ours. The following result actually holds in all dimensions.

\begin{teo}
We have $c^{\rm alm}_{\rm str}(M)=c(M)$ for every compact manifold $M$.
\end{teo}
\begin{proof}
Since a strict spine is a spine, we have $c^{\rm alm}_{\rm str}(M)\geqslant c^{\rm alm}(M)$. We now show the converse. Let $P$ be a minimal almost simple spine, \emph{i.e.}~with $c^{\rm alm}(M)$ vertices. We now construct a strict one without increasing the number of vertices.

If $P$ is not strict, the complement $M\setminus P$ has some redundant balls: for each such $B$, there must be a $(n-1)$-component of $P$ which is adjacent on $B$ to one side and to another component of $M\setminus P$ on the other (because $M$ is connected). By removing a small open $(n-1)$-ball from this $(n-1)$-component we get an almost simple spine with one ball less in its complement. After finitely many such removals we get an almost simple strict spine $Q\subset P$ with the same vertices as $P$. Therefore $c^{\rm alm}_{\rm str}(M)\leqslant c^{\rm alm}(M)$.
\end{proof}

Finally, note that we have $c_{\rm str}(M)> c(M)$ in some cases. That is, a manifold $M$ may not have a simple strict spine with $c(M)$ vertices: for instance, $S^2$ and $D^2$ do not have a simple strict spine at all, and hence $c_{\rm str}(S^2)=+\infty$; we also have $c_{\rm str}(S^3)=1$ (the abalone is a strict spine with one vertex) and $c_{\rm str}(S^n)=0$ for every $n\geqslant 4$ (a generalized Bing's house without vertices in dimension $n\geqslant 4$ is constructed in \cite{Mat:special}). A reasonable complexity should be zero on spheres, so we do not investigate $c_{\rm str}$ here.

\section{Drilling spheres, handles, surgery, and connected sums} \label{handles:section}
We can finally employ the techniques introduced in the previous sections to study how complexity changes under the most common topological operations.

\subsection{Drilling} \label{drilling:subsection}
We consider first the effect of drilling along curves.
Let $M$ be a compact manifold and $\gamma\subset M$ a properly embedded $1$-manifold. The \emph{drilled} manifold $M_\gamma$ is $M_\gamma = \overline{M\setminus R(\gamma)}$ for some regular neighborhood $R(\gamma)$.  We have the following.
\begin{teo} \label{drilling:teo}
Let $M^n$ be a compact manifold of dimension $n\geqslant 4$, and $\gamma \subset M^n$ a properly embedded $1$-manifold. We have $c(M^n_\gamma)\leqslant c(M^n)$.
\end{teo}
\begin{proof}
Let $P$ be a minimal spine of $M$. If $\gamma$ lies in $\interior M$, we can isotope it inside $P$ and disjoint from the $1$-skeleton (because $\dim P \geqslant 3$). By drilling $P$ along $\gamma$ we thus get a spine $Q$ for $M_\gamma$. Lemma \ref{bubble:lemma} implies that $Q$ has no more vertices than $P$, and hence $c(M_\gamma)\leqslant c(M)$.

If $\gamma$ intersects $\partial M$, we can isotope it so that $\gamma=\gamma'\cup \lambda$ where $\gamma'$ lies in $P$ as before and $\lambda$ consists of arcs connecting $P$ to $\partial M$, and each arc is a fiber of a collar of $P$, see Fig. \ref{arco:fig}.

\begin{figure}
 \begin{center}
  \includegraphics[width = 4cm]{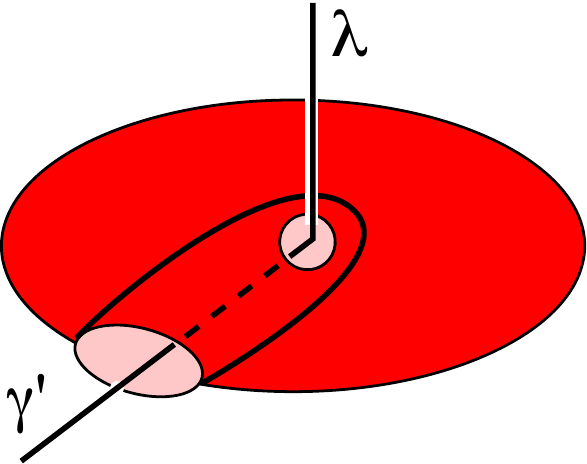}
 \end{center}
\caption{Put $\gamma$ as $\gamma'\cup\lambda$ with $\gamma'\subset P$ and $\lambda$ vertical. Then drill along $\gamma'$ and make a hole around $\lambda$.}
\label{arco:fig}
\end{figure}

A spine $P'$ of $\overline{M\setminus R(\gamma')}$ is constructed from $P$ by drilling along $\gamma'$. The polyhedron $P'$ intersects $\lambda$ transversely in some points. An almost simple spine of $M_\gamma$ is $P'$ with some small open balls removed around these points. By Theorem \ref{almost:simple:teo} we get $c(M^n_\gamma)\leqslant c(M^n)$.
\end{proof}
The condition $n\geqslant 4$ is necessary: in dimension 3 there is no general estimate relating $c(M_\gamma)$ and $c(M)$. If $M$ is closed and $c(M)>0$, it is always possible to decrease the complexity by some appropriate drilling. 
\begin{teo} \label{drilling:strict:teo}
Let $M^n$ be a closed manifold with $c(M^n)>0$ and $n\geqslant 2$. There is a simple closed curve $\gamma\subset M^n$ such that $c(M^n_\gamma)<c(M^n)$.
\end{teo}
\begin{proof}
Let $P$ be a minimal spine of $M$. Since $c(M)>0$, it has at least one vertex $v$. Let $C$ be a $(n-1)$-component of $P$ incident to $v$. By removing a small open $(n-1)$-ball from $C$ and then collapsing as more as possible we get an almost simple polyhedron $Q\subset P\subset M$ with strictly less vertices than $P$, since $v$ has been ``killed'' during the collapse.

The component $C$ is adjacent to one or two distinct components of $M\setminus P$. Each such component is an $n$-ball. If it is adjacent to two distinct balls, these glue to form a single ball in $M\setminus Q$, and hence $Q$ is a spine of $M$: a contradiction, since it has less vertices than $P$.  
Therefore $C$ is adjacent to a single ball, and $Q$ is a spine of $M_\gamma$ where $\gamma\subset M$ is a closed curve intersecting $P$ transversely in one point of $\sigma$. See an example in Fig.~\ref{buca:fig}.

\begin{figure}
 \begin{center}
   \resizebox{12.5 cm}{!}{\input{buca.pdftex_t}}
 \end{center}
\caption{By making a hole on a $(n-1)$-component $C$ of $P$ we get a spine of $M\gamma$ with $\gamma$ intersecting $P$ transversely in one point. After collapsing, we kill all the vertices adjacent to $C$.}
\label{buca:fig}
\end{figure}
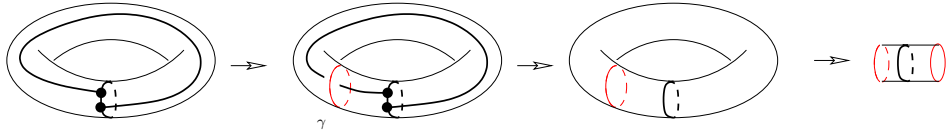

\end{proof}
In other words, every closed manifold of positive complexity is obtained by filling a manifold of strictly smaller complexity. 

\begin{rem} In contrast with the previous result, Theorem \ref{drilling:strict:teo} also holds in dimension 3. For instance, a lens space (whic may have arbitrarily high complexity) is obtained by filling a solid torus, which has complexity zero. The Matveev-Fomenko-Weeks smallest closed hyperbolic 3-manifold has complexity 9 \cite{Mat, MaPe} and can be obtained by filling the figure-eight knot sibling, which has complexity 2 \cite{CaHiWe}. Note that the hyperbolic volume (and hence Gromov norm) satisfies the \emph{opposite} inequality ${\rm Vol} (M_\gamma^3) > {\rm Vol} (M^3)$ for any $\gamma$.
\end{rem}

\subsection{Handles}
We have the following.
\begin{teo} \label{handles:teo} 
Let $N^n$ be obtained from $M^n$ by adding a handle of index $i$.
\begin{itemize}
\item If $i=n$ and $n\geqslant 3$, we have $c(N^n)=c(M^n)$;
\item If $i=n-1$ and $n\geqslant 4$, we have $c(N^n)\geqslant c(M^n)$;
\item if $i<n-1$, we have $c(N^n)\leqslant c(M^n)$.
\end{itemize}
\end{teo}
\begin{proof}
Suppose $i=n$. A spine $P$ of $M$ is also a spine of $N$: hence $c(N)\leqslant c(M)$. Conversely, we can easily construct a spine of $M$ from a spine $P$ of $N$ with the same number of vertices by drilling around a point of $P$ not contained in the $1$-skeleton (which exists since $n\geqslant 3$), see Lemma \ref{bubble:lemma}. 

Suppose $i=n-1$. The inverse operation of attaching a $(n-1)$-handle is drilling along the cocore arc of the handle: the inequality follows from Theorem \ref{drilling:teo}. 

We are left with the case $i<n-1$. Let $P$ be a minimal spine of $M$, that is a spine with $c(M)$ vertices. By using a collar for $P$ (see Section \ref{cut:subsection}), we attach the core disc $D^i$ of the handle directly to $P$ and get a spine $P\cup D^i$ of $N$. Then $c(N)\leqslant c(M)$ by Lemma \ref{low:attach:lemma}.
\end{proof}

\subsection{Drilling along spheres}
In Section \ref{drilling:subsection} we showed the effect of drilling along a curve. Drilling along higher-dimensional spheres gives the opposite inequality. As above we set $M_S = \overline{M\setminus R(S)}$. In the following, the sphere $S\subset M$ is not necessarily locally flat.

\begin{cor} \label{spheres:cor}
Let $M$ be a manifold and $S\subset\interior M$ a $k$-sphere with $k\geqslant 2$. We have $c(M_S)\geqslant c(M)$.
\end{cor}
\begin{proof}
The manifold $M$ is obtained from $M_S$ by adding a $(n-k)$-handle and a $n$-handle: the result then follows from Theorem \ref{handles:teo}. To prove the first assertion, represent $R(S)$ as a block bundle \cite{block} over $S$ (block bundles play the r\^ole of normal bundles in the PL category). A $n$-block over a $k$-simplex $\sigma$ of $S$ is a $(n-k)$-handle $H$. The complement $\overline{R(S)\setminus H}$ collapses onto the disc $\overline{S\setminus\sigma}$, and hence to a point: it is therefore a disc, \emph{i.e.}~a $n$-handle. 
\end{proof}

\subsection{Surgery}
Let $M^n$ be a manifold of dimension $n\geqslant 2$. A \emph{surgery} along a simple closed curve $\gamma\subset M^n$ whose regular neighborhood is homeomorphic to $D^{n-1}\times S^1$ consists of substituting this regular neighborhood with $S^{n-2}\times D^2$ via some gluing map on the boundaries, both homeomorphic to $S^{n-2}\times S^1$.

\begin{cor} \label{surgery:cor}
Let $n\geqslant 4$ and $N^n$ be obtained from a closed $M^n$ via surgery along some closed curve. We have $c(N^n)\leqslant c(M^n)$. If $c(M^n)>0$, there is a closed curve yielding $c(N^n)< c(M^n)$.  
\end{cor}
\begin{proof}
A surgery consists of drilling along the curve, and then adding a $2$-handle and a $n$-handle. Therefore the result follows from Theorems \ref{drilling:teo}, \ref{drilling:strict:teo}, and \ref{handles:teo}.
\end{proof}

\begin{cor}
Every closed manifold $M^n$ of dimension $n\geqslant 4$ can be transformed into a manifold with complexity zero after at most $c(M)$ surgeries along simple closed curves.
\end{cor}

When $M$ is simply connected, a surgery is just a connected sum with $S^2\times S^{n-2}$. In dimension $n\neq 4$ there is no simply connected manifold $M$ of positive complexity, see Theorem \ref{simply:connected:teo}. We do not know if there is one such manifold in dimension 4. If so, the following holds.

\begin{cor} \label{sum:cor}
If $M^4$ is closed simply connected and $c(M^4)>0$, then
$$c\big(M^4\# (S^2\times S^{2})\big)<c(M^4).$$
\end{cor}

\subsection{Connected sums}
Complexity is subadditive with respect to connected sums.

\begin{teo} \label{sum:teo}
Let $M^n\# N^n$ be obtained from $M^n$ and $N^n$ via (boundary) connected sum. If $n\geqslant 3$ we have 
$$c(M^n\#N^n)\leqslant c(M^n)+c(N^n).$$
\end{teo}
\begin{proof}
Making a connected sum corresponds to removing two $n$-handles from the (disconnected) manifold $M_1\sqcup M_2$, adding one $1$-handle and one $n$-handle. Noone of these operations can increase the complexity when $n\geqslant 3$. Similarly, a $\partial$-connected sum is the addition of one $1$-handle.
\end{proof}
Complexity is actually additive on connected sums in dimension $n=3$, see \cite{Mat}. Actually, we do not know any example of non-additivity in higher dimension. If there were a closed simply connected 4-manifold with $c(M)>0$, then Corollary \ref{sum:cor} would yield a non-additive connected sum.

\section{Coverings and products} \label{coverings:section}
We study how complexity changes under finite coverings and products. 
\subsection{Coverings}
We have the following.

\begin{teo} \label{covering:teo}
Let $p:\widetilde M\to M$ be a $d$-sheeted covering. We have $c(\widetilde M)\leqslant d\cdot c(M)$.
\end{teo}
\begin{proof}
Let $P$ be a minimal spine of $M$. Then $p^{-1}(P)$ is a simple spine of $\widetilde M$ with $d\cdot c(M)$ vertices.
\end{proof}

In contrast with Gromov norm, we often get a strict inequality $c(\widetilde M)<dc(M)$. For instace, lens spaces may have arbitrarily high complexity while the complexity of their universal covering $S^3$ is zero. The following consequence is worth mentioning.

\begin{cor} \label{zero:covering:cor}
If $M$ has complexity zero, every finite covering of $M$ has complexity zero.
\end{cor}

\subsection{Products}
We do not know whether there is some general inequality which relates the complexity $c(M\times N)$ of a product with the complexities $c(M)$ and $c(N)$ of the factors. However, we have the following.

\begin{teo} \label{product:teo}
Let $M^m$, $N^n$ be compact manifolds with $m,n\geqslant 1$. If $M$ has boundary then 
$$c(M^m\times N^n)=0.$$
\end{teo}
\begin{proof}
Let $P$ be any strict simple spine of $M$ (no balls in the complement, \emph{i.e.}~$M$ collapses onto $P$, see Section \ref{strict:subsection}). Then $M\times N$ collapses onto $P\times N$. Moreover, $P\times N$ is simple without vertices: if $x\in P$ is of type $k$, a point $(x,y)\in P\times N$ is of type $k+n>0$. Therefore $c(M\times N)=0$.
\end{proof}

\begin{cor} \label{sphere:product:cor}
We have $c(S^m\times N)=0$ for every $m\geqslant 2$ and every manifold $N$.
\end{cor}
\begin{proof}
If $N$ has boundary we are done by Theorem \ref{product:teo}. Otherwise, let $N'$ be $N$ with a ball removed. We have $c(S^m\times N')=0$, and $S^m\times N'$ is obtained from $S^m\times N$ by drilling along a $m$-sphere: Corollary \ref{spheres:cor} gives $c(S^m\times N)=0$.
\end{proof}

\section{Normal hypersurfaces} \label{normal:section}
When $n=3$, Matveev proved \cite{Mat} that complexity is non-increasing when cutting a $3$-manifold along an incompressible surface. This was done by putting the surface in normal position with respect to the handle decomposition induced by a minimal spine, and by showing that a simple spine can be ``cut'' along a normal surface. 

To extend this result, we define here normal discs in simplexes $\Delta$ of any dimension. These can be used to define normal hypersurfaces with respect to triangulations or simple spines in any dimension: we do this in the simple case where every simplex contains at most one normal disc. In the dual setting of simple spines, this means that the normal surface is actually a subpolyhedron of the spine. 

\subsection{Normal discs} \label{normal:subsection}
We extend the usual definition of normal discs in a tetrahedron, used in 3-dimensional topology, to all dimensions. Let $\Delta = \Delta^{n+1}$ be the $(n+1)$-simplex.

\begin{defn} A \emph{normal disc} in $\Delta$ is a subpolyhedron dual to some partition $\calP = \{V_1, V_2\}$ of the vertices of $\Delta$ into two non-empty subsets, see Definition \ref{dual:defn}.
\end{defn}

By Proposition \ref{Z:prop} a normal disc is homeomorphic to $\Pi^n_n\cong D^n$ and is hence indeed a disc. It is a subcomplex of the subdivided $\Delta'$. The \emph{type} of a normal disc is the unordered pair $(\# V_1,\# V_2)$. 

\begin{rem}
There are $\left(\begin{array}{c} n+2 \\ \# V_1 \end{array}\right)$ distinct normal discs of type $(\# V_1,\# V_2)$, except when $\# V_1 = \# V_2$: in this case there are half of them.
\end{rem}

\begin{figure}
 \begin{center}
  \includegraphics[width = 10cm]{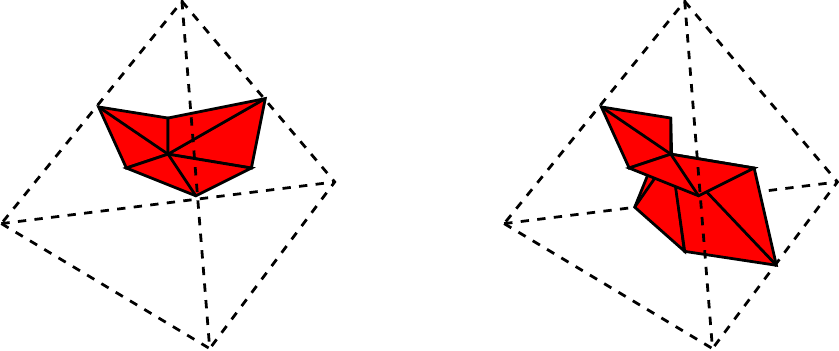}
 \end{center}
\caption{In dimension 3 the normal discs are the usual normal triangles and squares. Each is a subcomplex of $\Pi^2$, dually embedded in the tetrahedron $\Delta$.}
\label{normal:fig}
\end{figure}

\begin{rem}
The 3-simplex contains the usual 7 normal discs: $4$ normal triangles of type $(3,1)$ and $3$ squares of type $(2,2)$, see Fig.~\ref{normal:fig}. The 4-simplex contains 15 normal discs: $5$ normal tetrahedra of type $(4,1)$ and $10$ normal dipyramids (two tetrahedra attached along one face) of type $(3,2)$.
\end{rem}

\subsection{Homology}
The homology group $H_n(P^n;\matZ_{2\matZ})$ of a simple polyhedron $P^n$ can be naturally interpreted as the set of all closed $n$-submanifolds of $P$. We prove this fact first locally and then globally.

Represent $\Pi^n\subset\Delta$ dually to $\Delta=\Delta^{n+1}$. Every normal disc $D$ is contained in $\Pi^n$.
Its boundary $\partial D = D\cap\partial \Pi^n$ is a $(n-1)$-sphere. Consider the maps
$$
\begin{CD} 
\left\{ \begin{array}{c} {\rm Normal} \\ {\rm discs\ in\ } \Delta \end{array} \right\} @>\partial>>
\left\{ \begin{array}{c} {\rm Closed\ submanifolds\ in} \\ \partial \Pi^n {\rm \ of\ dimension\ } n-1
\end{array} \right\} @> [.]>>
H_{n-1}(\partial\Pi^n;\matZ/_{2\matZ}).
\end{CD}
$$
Of course $[.]$ sends a manifold $\Sigma$ to its class $[\Sigma]$.
\begin{prop} \label{homology:prop}
If $n\geqslant 2$, both maps are bijections.
\end{prop}
\begin{proof}
Since $\dim\partial\Pi^n=n-1$, a cycle $\alpha\in H_{n-1}(\partial\Pi^n;\matZ/_{2\matZ})$ is represented by a unique subpolyhedron $Y\subset \partial\Pi^n\subset S^n$. We must prove that $Y=\partial D$ for a normal disc $D$, induced by some partition of $V$ into two sets.
 
The polyhedron $Y$ is the closure of the union of some $(n-1)$-components of $\partial\Pi^n$. Since $H_{n-1}(S^n;\matZ/_{2\matZ})\cong H^1(S^n;\matZ_{2\matZ})=\{0\}$ for $n\geqslant 2$, the connected components of $S^n\setminus Y$ may in fact be partitioned into two sets in a unique way, such that every $(n-1)$-component of $Y$ is adjacent to components belonging to distinct sets. The vertices of $\Delta$ are thus partitioned in two sets, and $Y$ is the normal disc dual to the partition, as required.
\end{proof}

\begin{cor} Let $P^n$ be a simple polyhedron. Closed $n$-submanifolds of $P$ are in natural bijection with $H_n(P^n,\matZ/_{2\matZ})$. \end{cor}
\begin{proof}
Take a triangulation of $P$. Every cycle is represented by a subcomplex $S\subset P$. It intersects the link of every point $x$ in $P$ into a cycle: by Proposition \ref{homology:prop} this is a sphere and the star of $S$ in $x$ is a normal disc. Therefore $S$ is submanifold. Different submanifolds yield different subcomplexes and hence different cycles.
\end{proof}

\subsection{Cutting along normal surfaces}
Let $S\subset M$ be a closed submanifold in $M$ of codimension 1. As above, we set $M_S = \overline{M\setminus R(S)}$. A key property of 3-dimensional complexity, proved by Matveev in \cite{Mat}, is that $c(M_S)\leqslant c(M)$ whenever $S$ is an incompressible surface in a 3-manifold $M$. We prove here a kind of generalization of this fact.

The powerful notion of incompressible surface unfortunately does not extend easily to higher dimensions. On the other hand, in dimension 3, every homology class in $H_2(M,\matZ)$ is represented by an incompressible surface. We propose the following generalization of Matveev's result.

\begin{teo} Let $M^n$ be a compact manifold. Every element in $H_{n-1}(M^n;\matZ/_{2\matZ})$ is represented by a submanifold $S$ such that $c(M^n_S)\leqslant c(M^n)$.
\end{teo}
\begin{proof}
Let $Q$ be a minimal spine of $M$. The map 
$$i_*:H_{n-1}(Q;\matZ/_{2\matZ})\to H_{n-1}(M;\matZ/_{2\matZ})$$ 
is surjective, because $M\setminus Q$ consists of balls and a collar of $\partial M$. Proposition \ref{homology:prop} then implies that every cycle $\alpha$ is represented by a closed submanifold $S\subset Q$.

Let $P$ be obtained from $Q$ by drilling along $S$, see Section \ref{drilling:section}.
The polyhedron $P$ is a simple spine of $M_S$: however, $S$ intersects the $1$-skeleton of $Q$, so we cannot use Lemma \ref{bubble:lemma} to conclude that $c(M_S)\leqslant c(M)$.

\begin{figure}
 \begin{center}
 \includegraphics[width = 12.5cm]{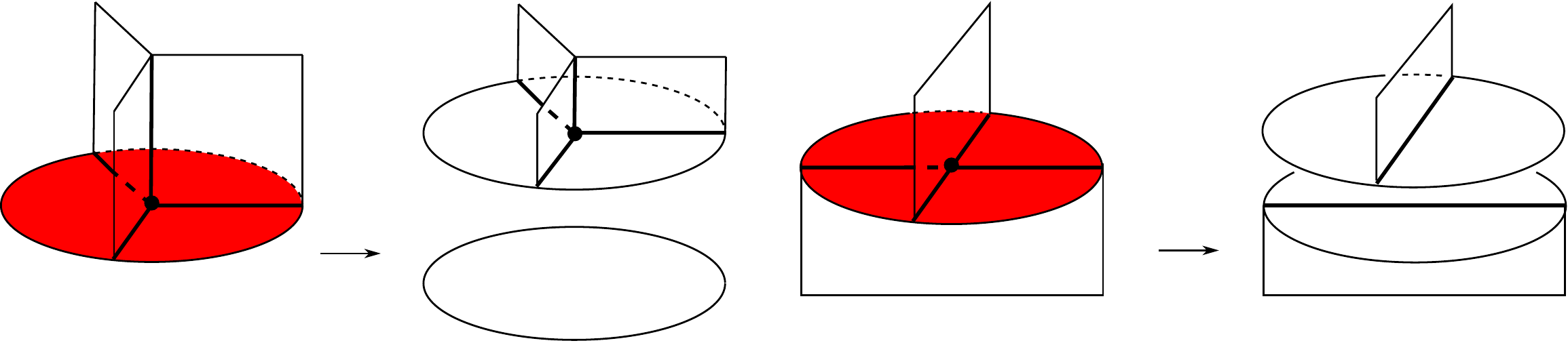}
 \end{center}
\caption{When the normal disc $D$ is of type $(n,1)$, there is an adjacent $1$-component not lying in $D$. Therefore the drilling deletes a vertex and produces a new one (left). If $D$ is of some other type, it contains all the adjacent $1$-components: therefore a vertex is deleted and noone is created (right). Here, $n=2$.}
\label{hypersurface:fig}
\end{figure}

The manifold $S$ is the union of the colsure of some $(n-1)$-components of $Q$. Let $R$ be a small regular neighborhood of $S$. The proof of Lemma \ref{bubble:lemma} shows that the new vertices of $P$ lie in the transverse intersection of the hypersurface $\partial R$ with the $1$-skeleton of $Q$. There is precisely one such intersection for every pair $(v,e)$ such that $v$ is a vertex of $Q$ contained in $S$ and $e$ is an oriented edge (\emph{i.e.}~1-component) exiting from $v$ and not contained in $S$. Here $\partial R$ intersects $e$ transversely near $v$.

Let $(v,e)$ be one such pair. The star of $v$ in $(M,Q)$ is homeomorphic to $(\Delta^n,\Pi^{n-1})$ and it intersects $S$ into a normal disc $D\subset\Pi^{n-1}$, determined by some partition $\calP = \{V_1,V_2\}$ of the vertices of $\Delta^n$. The edge $e$ is dual to a facet of $\Delta^n$. Since $e\not\subset D$, all the $n$ vertices lying in this facet belong to the same set of the partition. Therefore the normal disc is of type $(n,1)$ as in Fig~\ref{hypersurface:fig}-left. As we see in Fig.~\ref{hypersurface:fig}-left, one new vertex is created but $v$ is destroyed: the total number of vertices in the spine does not increase.
\end{proof}

\section{Nerve} \label{nerve:section}
We study here the nerve $\calN$ of a pair $(M,P)$ consisting of a closed manifold $M$ and a simple spine $P\subset \interior M$, see Section \ref{nerve:subsection}. We prove that the nerve map $\varphi:M\to |\calN|$ induces an isomorphism on fundamental groups when $\pi_1(M)$ has no torsion. Therefore $\calN$ carries many topological informations on $M$.

We have $\dim |\calN|\leqslant \dim M$, and $\dim |\calN| <n$ precisely if $P$ has no vertices. 
Actually, the $n$-dimensional part of $|\calN|$ has a kind of singular triangulation (see Remark \ref{singular:rem}), with one singular $n$-simplex ``dual'' to each vertex of $P$. 

The facts listed in Subsection \ref{basic:nerve:subsection} will be used to prove most of the results stated in Sections \ref{homotopy:section}, \ref{norm:section}, and \ref{riemannian:section}.
The informations collected in Sections \ref{fundamental:subsection} and \ref{singular:subsection} are only needed to prove Theorem \ref{norm:teo}.

\subsection{Basic properties} \label{basic:nerve:subsection}
We will need the following lemma.
\begin{lemma} \label{component:lemma}
Let $P$ be a spine of a compact $M$. Let $C$ be a component of $(M,P)$. The image of $i_*:\pi_1(C)\to\pi_1(M)$ is either finite or has a finite-index subgroup contained in the image of $i_*:\pi_1(N)\to\pi_1(M)$ for some boundary component $N\subset \partial M$.
\end{lemma}
\begin{proof}
Consider the cut map  $f:M_P\to M$, see Section \ref{cut:subsection}. Let $\widetilde C$ be a connected component of $f^{-1}(C)$. The restriction of $f$ to $\widetilde C$ is a finite covering $f:\widetilde C\to C$. We have the following commutative diagram (with appropriate basepoints):
$$
\begin{CD} 
\pi_1(\widetilde C) @>i_*>> \pi_1(M_P)\\ 
@VV{f_*}V @VV{f_*}V\\ 
\pi_1(C) @>i_*>> \pi_1(M) 
\end{CD}
$$
By Proposition \ref{cut:prop}, every component of $M_P$ is either a disc or a product which maps to a collar of a component $N$ of $\partial M$. Therefore $(f_*\circ i_*)(\pi_1(\widetilde C))$ is either trivial or contained in $i_*(\pi_1(N))$. Since $f:\widetilde C\to C$ is a covering, the subgroup $f_*(\pi_1(\widetilde C))$ has finite index in $\pi_1(C)$: therefore $i_* (\pi_1(C))$ is a finite extension of 
$(i_*\circ f_*)(\pi_1(C)) = (f_*\circ i_*)(\pi_1(\widetilde C))$. 
\end{proof}

Following Gromov \cite{Gro}, a set $X\subset M^n$ in a manifold $M^n$ is \emph{amenable} if for every path-connected component $X'$ of $X$ the image of the map $i_*:\pi_1(X')\to\pi_1(M)$ is an amenable group (see \cite{Gro} for a definition). Simple examples of amenable groups are finite and abelian groups. Subgroups and finite extensions of amenable groups are amenable. We have the following.

\begin{prop} \label{fiber:prop}
Let $P$ be a spine of a compact $M$. Let $\varphi:M\to |\calN|$ be a nerve map of $(M,P)$.
\begin{enumerate}
\item If $M$ is closed or with amenable boundary, the fiber $\varphi^{-1}(x)$ of every point $x\in|\calN|$ is amenable.
\item If $M$ is closed and $\pi_1(M)$ is torsion-free, the map $\varphi_*:\pi_1(M)\to\pi_1(|\calN|)$ is an isomorphism.
\end{enumerate}
\end{prop}
\begin{proof}
The fiber $\varphi^{-1}(x)$ of a point is contained in some component of $(M,P)$, since each fiber of the pre-nerve map does, see Section \ref{nerve:subsection}. Since subgroups and finite extensions of amenable groups are amenable, point (1) follows from Lemma \ref{component:lemma}.

We turn to (2). The map $\varphi:\calM\to |\calN|$ is surjective and the fiber $\varphi^{-1}(x)$ over a point of $|\calN|$ is path-connected: these two facts imply that $\varphi_*$ is surjective. On the other hand, Lemma \ref{component:lemma} implies that $i_*(\varphi^{-1}(x))<\pi_1(M)$ is a finite group. If $\pi_1(M)$ has no torsion, this finite group is trivial: this easily implies that $\varphi_*$ is injective.
\end{proof}

\begin{figure}
 \begin{center}
 \includegraphics[width = 8 cm]{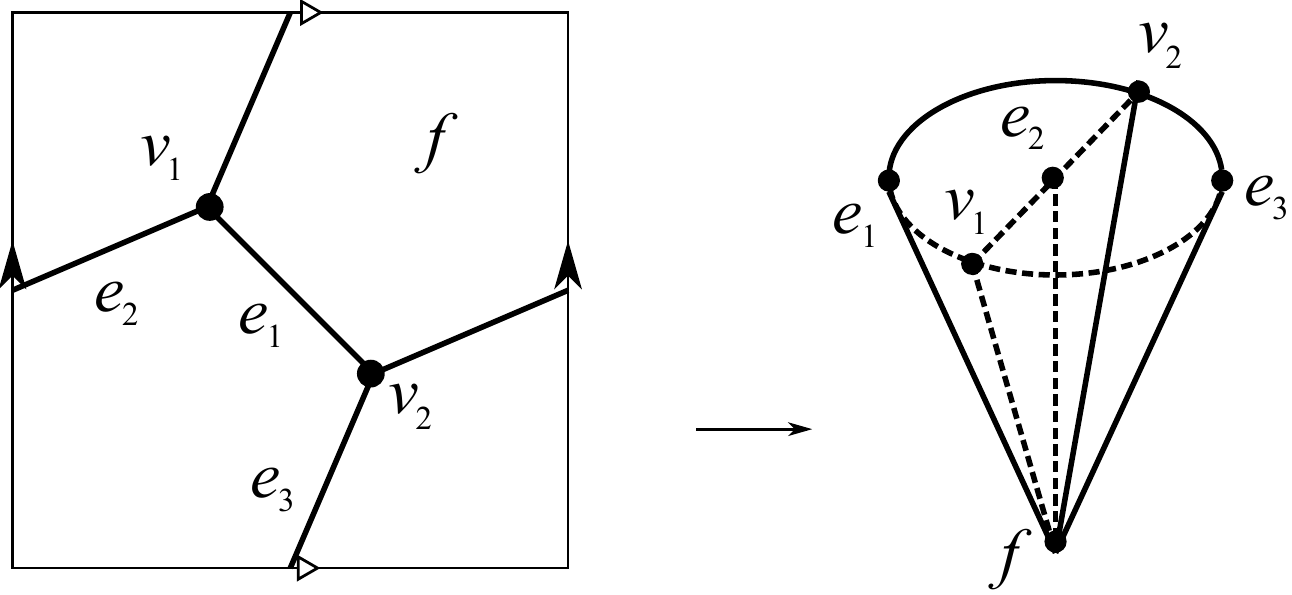}
 \end{center}
\caption{The pre-nerve of a spine $P$ of the torus $T$ with two vertices is homeomorphic to the cone over $P$, and is hence simply connected: the vertex of the cone is the $2$-component $f=T\setminus P$, which is incident to all the other components.}
\label{nerve0:fig}
\end{figure}

\begin{rem} We note that Proposition \ref{fiber:prop}-(2) does not hold in general for the pre-nerve map $\varphi_0:M\to |\calN_0|$: we really need Stein factorization here. See Exercise \ref{nerve:ex} and Fig.~\ref{nerve0:fig}.
\end{rem}

\begin{rem}
Some hypothesis on $\pi_1(M)$ is indeed necessary in Proposition \ref{fiber:prop}-(2). For instance, a hyperplane $H\subset\matRP^n$ is a spine of $\matRP^n$, but the nerve of $(\matRP^n,H)$ is a segment: hence $\varphi_*:\matZ/_{2\matZ}\to\{e\}$ is not an isomorphism.
\end{rem}

\begin{rem}
Let $P$ be a spine of a closed manifold $M$ dual to some triangulation $T$ as in Section \ref{dual:subsection} (with the discrete partition $v\sim v' \Leftrightarrow v=v'$) . The pre-nerve map $\varphi:T'\to\calN_0$ is an isomorphism. Therefore $\calN\cong\calN_0'\cong T''$ and the nerve map $\varphi:M\to |\calN|$ is a homeomorphism.
\end{rem}

Concerning the dimension of the nerve, we have the following.
\begin{prop} \label{dimension:nerve:prop}
Let $\calN$ be the nerve of any pair $(M,P)$.
We have $\dim\calN \leqslant \dim M$. We have $\dim\calN=\dim M$ if and only if $P$ has vertices.
\end{prop}
\begin{proof}
Set $n=\dim M$.
We have $\dim\calN = \dim\calN_0$. A $k$-simplex in $\calN_0$ is determined by some components $C_0< \ldots < C_k$. In particular, we have $0\leqslant  \dim C_0<\ldots <\dim C_k\leqslant n$. This implies that $\dim\calN_0\leqslant n$. If $P$ has no vertices, then $\dim C_0 \geqslant 1$ and hence $\dim\calN_0<n$. If $P$ has a vertex $v$, there is a chain $\{v\} = C_0<\ldots < C_n$ and hence $\dim\calN = n$.
\end{proof}

\subsection{Fundamental class} \label{fundamental:subsection}
Let $P\subset \interior M$ be any simple polyhedron in a compact manifold (possibly with boundary). Let  $\varphi:M\to |\calN|$ be a nerve map of $(M,P)$. We show that $\calN$ looks like a pseudomanifold. In particular, we can define the notion of fundamental class. Let us start with the following.
\begin{prop} \label{0:or:2:prop}
Every $(n-1)$-simplex in $\calN$ is adjacent to either zero or two $n$-simplexes.
\end{prop}
\begin{proof}
Let $\sigma$ be a $(n-1)$-simplex in $\calN$. The counterimage $\varphi^{-1}(\sigma_*)$ of its barycenter $\sigma_*$ is either a point or a connected 1-manifold in $\interior M$. If it is a point or a segment then $\sigma$ is adjacent to two $n$-simplexes. If it is a circle then $\sigma$ is not adjacent to any $n$-simplex.
\end{proof}

The nerve $\calN$ looks like a pseudomanifold. Note however that the adjacencies of the $n$-simplexes along the $(n-1)$-simplexes do not necessarily form a connected graph: for instance, $|\calN|$ might consist of a bouquet of two $n$-manifolds. 

A \emph{fundamental class} for $\calN$ is an element of $H_n(|\calN|,\matZ)$ which may be written (in simplicial homology) as a sum of \emph{all} $n$-simplexes of $\calN$, each with an appropriate sign $\pm 1$. We indicate such a class as $[\calN]$. A fundamental class induces an orientation on each $n$-simplex of $\calN$.

\begin{prop} \label{fundamental:prop}
If $M$ is oriented then $[\calN] = \varphi_*([M,\partial M])$ is a fundamental class.
\end{prop}
\begin{proof}
Let $T$ be the sufficiently subdivided triangulation of $(M,P)$ defining the nerve map $\varphi: T'\to\calN$. Since $\varphi$ is surjetive with connected fibers,
every $n$-simplex $\sigma$ in $\calN$ is the image of precisely one $n$-simplex $\eta$ of $T'$. 
The class $[M,\partial M]$ is fundamental class, \emph{i.e.}~it is represented by a sum of all the simplexes of $M$ with appropriate signs, and hence (using simplicial homology) also $\varphi_*([M,\partial M])$ is a sum of all the simplexes of $\calN$ with appropriate signs. 
\end{proof}

\subsection{Singular simplexes} \label{singular:subsection}
Let $P\subset \interior M$ be any simple polyhedron in a manifold and $\varphi:M \to |\calN|$ a nerve map. We have just seen that $\calN$ looks like a pseudomanifold. We now prove that the $n$-dimensional part of $\calN$ is the twice barycentric subdivision of a singular triangulation (see Remark \ref{singular:rem}), with one singular simplex corresponding ``dually'' to each vertex of $P$.

Let $v$ be a vertex of $P$. A \emph{local component} at $v$ is a component of a fixed open star of $v$ in $(M,P)$. The open star is homeomorphic to $\interior{\Delta,\Pi^{n-1}}$. In the dual representation, the polyhedron $\Pi^{n-1}$ is a subcomplex of $\Delta'$ and the components of $\interior{\Delta,\Pi^{n-1}}$ are naturally identified with the vertices of $\Delta'$. 

Every local component is contained in a unique component of $(M,P)$, so we get a simplicial map $\beta_v:\Delta' \to \calN_0$. The map $\beta_v$ is topologically a singular simplex in $|\calN_0|$. It sends the barycenter of $\Delta$ to the component $\{v\}$. The image of $\beta_v$ is precisely the star of $\{v\}$ in $\calN_0$.

We want a singular simplex in $|\calN|$. We therefore try to lift $\beta_v$ to $\calN$ along the projection $g:\calN\to\calN_0'$. The following result says that there is a natural way do to this.

\begin{prop}
There is a natural simplicial map $\alpha_v:\Delta'' \to \calN$ such that $g\circ\alpha_v = \beta_v'$.
\end{prop}
\begin{proof}
We can formalize this as follows. The restriction of $\varphi_0$ to the (closed) star of $v$ in $T$ splits naturally into two simplicial maps
$$
\begin{CD} 
\st(v,T) @> \varphi_v>> \Delta' @> \beta_v>> \calN_0.
\end{CD}
$$
The map $\varphi_v$ sends every vertex to the local component to which it belongs. The map $\varphi_v$ has connected fibers. Therefore there is a unique simplicial map $\alpha_v:\Delta''\to\calN$ which makes the following diagram commute.
$$
\xymatrix{ 
T' \ar[r]^{\varphi} & \calN \ar[r]^g & \calN_0' \\
\st(v,T)' \ar @{^{(}->} [u]^i \ar[r]_{\varphi'_v} & \Delta'' \ar @{.>}[u]^{\alpha_v} \ar[ur]_{\beta'_v}
}
$$
\end{proof}
The map $\alpha_v$ is actually well-defined only up to a permutation of the vertices of $\Delta$, induced from the chosen identification of the open star of $v$ with $\interior {\Delta,\Pi^{n-1}}$.

The following propositions show that the singular simplexes $\alpha_v$ at the varying of $v$ among the vertices of $P$ form a singular triangulation of the $n$-dimensional part of $|\calN|$ whose double barycentric subdivision yields the original triangulation of $\calN$.

\begin{prop} 
If $v\neq v'$ then $\alpha_v(\interior \Delta) \cap \alpha_{v'}(\interior\Delta)=\emptyset$.
\end{prop}
\begin{proof}
It suffices to prove the assertion for $\beta_v$ and $\beta_{v'}$. The image of $\interior\Delta$ along $\beta_v$ is the open star of $\{v\}$ in $\calN_0$. The open stars of $\{v\}$ and $\{v'\}$ are disjoint (because $v$ and $v'$ are not connected by an edge).
\end{proof}

\begin{prop} \label{injective:prop}
The map $\alpha_v$ is injective in $\interior {\Delta}$ for every $v$.
\end{prop}
\begin{proof}
We have a commutative diagram of topological maps
$$
\xymatrix{ 
M \ar[r]^{\varphi} & |\calN| \ar[r]^g & |\calN_0|\\
\st(v,T) \ar @{^{(}->} [u]^i \ar[r]_{\varphi_v} & \Delta \ar @{>}[u]^{\alpha_v} \ar[ur]_{\beta_v}
}
$$
Take $x\in\interior\Delta$. The image $\beta_v (x)$ lies in the interior of the star of $\{v\}$ in $\calN_0$. Then $\varphi_0^{-1}(\beta_v(x))$ lies in the interior of the star of $v$ in $T$. Therefore $\varphi^{-1}(\alpha_v(x))$ is also contained in $\st(v,T)$, and is connected since $\varphi$ has connected fibers.

Therefore $\varphi_v(\varphi^{-1}(\alpha_v(x)) = \alpha_v^{-1}(\alpha_v(x))$ is connected. The map $\alpha_v$ is finite-to-one by construction, thus this set is also finite. Therefore it consists of only one point $\{x\}$.
\end{proof}

\begin{prop} 
Every $n$-simplex of $\calN$ is contained in $\alpha_v(\Delta)$ for some $v$.
\end{prop}
\begin{proof}
Let $\sigma$ be a $n$-simplex of $\calN$.
Since the nerve map $\varphi:T'\to\calN$ has connected fibers, the preimage of $\sigma$ consists of a single $n$-simplex $\varphi^{-1}(\sigma)$.  We show that $\varphi^{-1}(\sigma)$ lies in the star $\st(v,T)$ of some vertex $v$: this implies that $\sigma$ is contained in $\alpha_v(\Delta)$.

Let $\bar\sigma$ be the $n$-simplex of $T$ which contains $\varphi^{-1}(\sigma)$. Consider the maps
$$
\begin{CD} 
T' @>\varphi>> \calN @>g>> \calN_0'.
\end{CD}
$$
The image $\varphi_0(\bar\sigma) = g(\varphi(\bar\sigma))$ is a $n$-simplex because $g$ preserves the dimension of simplexes. That is, the vertices of $\bar\sigma$ lie in distinct components $C_0<\ldots <C_n$ of $(M,P)$. This implies that $\dim C_i=i$, and thus $C_0=\{v\}$ is a vertex. Therefore $\bar\sigma$ lies in $\st(v,T)$.
\end{proof}

Finally, we show that the singular simplexes glue in pairs along their facets. An \emph{edge} of $P$ is a 1-component not homeomorphic to a circle. An edge $e$ connects two (possibly coinciding) vertices $v_1$ and $v_2$, see Fig~\ref{connecting:fig}. Identify the stars of $v_1$ and $v_2$ with $(\Delta, \Pi^{n-1})$. The edge $e$ intersects each $\Delta$ in the barycenter of a facet $f_i$ of $\Delta$. The edge $e$ induces a simplicial isomorphism $\psi:f_1\to f_2$ defined by taking track along $e$ of the incident components. We have the following.

\begin{figure}
 \begin{center}
  \includegraphics[width = 10cm]{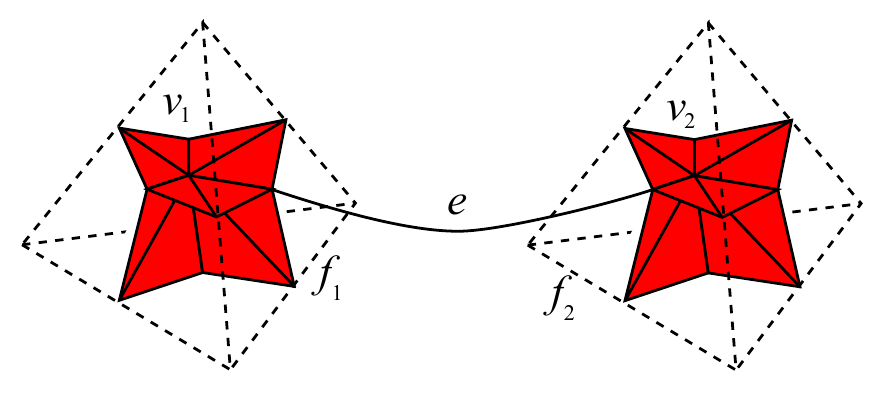}
 \end{center}
 \caption{An edge $e$ connecting two vertices $v_1$ and $v_2$. It intersects two facets $f_1$ and $f_2$ of the dual simplexes.}
 \label{connecting:fig}
\end{figure}

\begin{prop} \label{edge:prop}
We have $\alpha_{v_1}|_{f_1} = \alpha_{v_2}\circ\psi$.
\end{prop}
\begin{proof}
Take
$$ R = \overline{R(e,T) \setminus \big(R(v_1,T) \cup R(v_2, T)\big)}.$$
We can identify $R$ with a cylinder $\Delta^{n-1}\times [-1,1]$, which intersects $P$ as $\Pi^{n-2}\times [-1,1]$, such that $\varphi_0(x,t)$ does not depend on $t\in [-1,1]$ for every $x\in\Delta^{n-1}$. Therefore $\beta_{v_1}|_{f_1} = \beta_{v_2}\circ\psi$ and the lifts $\alpha_{v_1}|_{f_1}$ and $\alpha_{v_2}\circ \psi$ also coincide.
\end{proof}

Suppose $M$ is oriented. By Proposition \ref{fundamental:prop} the orientation on $M$ induces a fundamental class $[\calN]$ in $\calN$ and hence an orientation on its $n$-simplexes. A singular simplex $\alpha_v$ can thus be orientation-preserving or reversing, and we set $\sigma_v =0$ or $1$ correspondingly.

We would like to say that $\sum (-1)^{\sigma_v} \alpha_v$ is a cycle representing $[\calN]$. However, this argument is wrong: a \emph{singular} triangulation, without any additional data, does not yield a cycle, because boundary maps do not necessarily cancel algebraically (for that reason, the slightly richer notion of \emph{semisimplicial complex} \cite{semi} or \emph{$\Delta$-complex} \cite[Page  103]{Ha} is usually employed in the literature). We solve this problem by averaging each $\alpha_v$ over its $(n+1)!$ different representations. That is, we define
$$\tilde\alpha_v = \frac 1{(n+1)!}\sum_{p \in S_{n+1}} (-1)^{\sigma_v + {\rm sgn}(p)} \alpha_v\circ \theta_p$$
where $\theta_p:\Delta\to \Delta$ is the combinatorial isomorphism induced by the permutation $p$ of the vertices of $\Delta$ (which have a fixed ordering, \emph{i.e.}~a fixed identification with $\{1,\ldots, n+1\}$). The chain $\tilde \alpha_v$ depends only on $v$.

\begin{prop} \label{efficient:cycle:prop}
Let $M$ be oriented. We have
$$\varphi_*\big([M,\partial M]\big) = [\calN] = \left[ \sum_v \tilde\alpha_v\right].$$
\end{prop}
\begin{proof} First, note that $\sum_v \tilde\alpha_v$ is indeed a cycle: the terms in the boundary cancel in pairs thanks to Proposition \ref{edge:prop}. The twice subdivision of $\sum_v\tilde\alpha_v$ yields the fundamental class $[\calN]$, each simplex being counted $(n+1)!$ times with coefficient $1/(n+1)!$.
\end{proof}

\section{Homotopy type} \label{homotopy:section}
We study here the relations between the complexity and some homotopy invariants of a manifold. First, note that the complexity is not a homotopy invariant, since it distinguishes homotopically equivalent lens spaces: we have $c(L_{7,1})=4$ and $c(L_{7,2})=2$, see \cite{Mat}.

\subsection{Fundamental groups}
It might be that every simply connected manifold has complexity zero. However, this question is open in dimension 4.
\begin{teo} \label{simply:connected:teo}
A simply connected compact manifold of dimension $n\neq 4$ has complexity zero.
\end{teo}
\begin{proof}
A simply connected manifold of dimension $n\neq 4$ has a handle decomposition without 1-handles. In dimension 3, this follows from Perelman's proof of Poincar\'e Conjecture. Concerning dimension $n\geqslant 5$, see \cite[Lemma 6.15 and the subsequent remark]{PL}. By turning inside-out the handle decomposition, we get $M$ as $\partial M \times [0,1]$ plus some handles of index $\neq n-1$. The manifold has then complexity zero by Theorems \ref{product:teo} and \ref{handles:teo}.
\end{proof}
The smallest known fundamental group of a manifold with positive complexity is $\matZ/_{4\matZ}$: the lens space $L_{4,1}$ has $c(L_{4,1})=1$. We do not know if there are manifolds of positive complexity with fundamental group $\matZ/_{2\matZ}$ or $\matZ/_{3\matZ}$.

On the other hand, every finitely presented group is the fundamental group of a closed 4-manifold of complexity zero: see Corollary \ref{every:cor}. 

\subsection{Essential manifolds}
Following Gromov \cite{Gro2} a non-simply connected closed manifold $M$ is \emph{essential} if one of the following equivalent conditions holds:
\begin{enumerate}
\item the image $f_*([M])\in H_n(K(\pi,1))$ of its fundamental class in the corresponding Eilenberg-MacLane space is non-trivial;
\item there is no map $f:M\to X$ onto a lower-dimensional polyhedron $X$ which induces an isomorphism on fundamental groups.
\end{enumerate}
The equivalence between this definitions was proved by Babenko \cite{Ba}. A group 
is \emph{virtually torsion-free} if it contains a finite-index torsion-free subgroup. 

\begin{teo} \label{essential:teo}
Let $M$ be a closed manifold with virtually torsion-free infinite fundamental group. If $c(M)=0$ then $M$ is not essential.
\end{teo}
\begin{proof}
There is a finite covering $\widetilde M\to M$ such that $\pi_1(\widetilde M)$ is infinite and torsion-free. We have $c(\widetilde M)=0$ by Corollary \ref{zero:covering:cor}. Let $P$ be a spine of $\widetilde M$ without vertices. By Proposition \ref{dimension:nerve:prop}, the nerve $|\calN|$ of $(\widetilde M,P)$ has dimension smaller than $\dim \widetilde M$. On the other hand, the nerve map $\varphi:\widetilde M\to |\calN|$ yields an isomorphism on fundamental groups by Proposition \ref{fiber:prop}. Therefore $\widetilde M$ is not essential. This implies that $M$ is not essential.
\end{proof}
A manifold is \emph{aspherical} if $\pi_i(M)$ is trivial for all $i>1$. Equivalently, the universal cover $\widetilde M$ is contractible. Closed aspherical manifolds and real projective spaces are essential. The following fact is well-known.

\begin{prop}
An aspherical manifold has torsion-free fundamental group.
\end{prop}
\begin{proof}
Suppose $M$ is aspherical and $\pi_1(M)$ contains a non-trivial finite cyclic group $C_p$. Let $\widetilde M\to M$ be the covering determined by $C_p$. The manifold $\widetilde M$ is also aspherical. Therefore $H_i(\widetilde M,\matZ) \cong H_i(C_p,\matZ)$. However, the group $H_i(C_p,\matZ)$ is non-trivial for infinitely many values of $i$: a contradiction because $M$ is a manifold.
\end{proof}

\begin{cor} \label{aspherical:cor}
A closed aspherical manifold $M$ has $c(M)>0$.
\end{cor}
In particular, the $n$-torus $S^1\times \cdots \times S^1$ has positive complexity. Products of $S^1$ and/or closed surfaces of positive genus also have positive complexity. 
\begin{rem} 
Some hypothesis on $\pi_1(M)$ is necessary in Theorem \ref{essential:teo}, since $\matRP^n$ is essential and $c(\matRP^n)=0$ for $n\geqslant 2$.
\end{rem}
\begin{rem} The same arguments show that connected sums of aspherical manifolds have positive complexity. \end{rem}

\section{Gromov norm} \label{norm:section}
Let $M$ be an orientable manifold, possibly with boundary. The \emph{Gromov norm} $||M||$ of $M$ is
$$||M|| = ||[M,\partial M]|| = \inf \Big\{|a_1|+\ldots +|a_k|\ \Big|\ [M,\partial M] = \sum a_i\sigma_i\Big\}$$
where the infimum is taken among all representations of the fundamental class $[M,\partial M]\in H_n(M,\partial M;\matR)$ as singular cycles $\sum a_i\sigma_i$. See \cite{Gro}.

\subsection{Complexity zero}
\begin{teo} \label{zero:norm:teo}
Let $M$ be closed or with amenable boundary. If $c(M)=0$ then $||M||=0$.
\end{teo}
\begin{proof}
Since $c(M)=0$, $M$ has a spine $P$ without vertices. Let $\calN$ be the nerve of $(M,P)$.
We have $\dim \calN<\dim M$ by Proposition \ref{dimension:nerve:prop}. Proposition \ref{fiber:prop} implies that $M$ fibers on a low-dimensional polyhedron with amenable fibers. Therefore $M$ can be covered by at most $n$ amenable (not necessarily connected) open sets. If $M$ is closed, Gromov's Vanishing Theorem \cite{Gro} implies that $||M||=0$. 

If $M$ has boundary, every boundary component lies in one fiber and the fibration extends naturally to the double $DM$. Therefore $||DM||=0$. Since $\partial M$ is amenable, we have $||DM|| = 2||M||$ by \cite{Ku}, so $||M||=0$. 
\end{proof}
The hypothesis on $\partial M$ is indeed necessary: for instance, a $3$-dimensional handlebody of genus $2$ has complexity zero and positive norm (because its boundary has positive Gromov norm).

\subsection{General inequality}
\begin{teo} \label{norm:teo}
Let $M$ be closed with virtually torsion-free $\pi_1(M)$. Then 
$$||M||\leqslant c(M).$$
\end{teo}
\begin{proof}
Let $\widetilde M\to M$ be the degree-$d$ covering induced by the torsion-free subgroup. Since $||\widetilde M||=d||M||$ and $c(\widetilde M)\leqslant dc(M)$, it suffices to prove the theorem for $\widetilde M$, which we still call $M$ for simplicity.

Now $\pi_1(M)$ is torsion-free. Let $P$ be a minimal spine of $M$, \emph{i.e.}~a spine with $c(M)$ vertices. Let $\calN$ be the nerve of $(M,P)$. A nerve map $\varphi:M\to |\calN|$ induces an isomorphism $\varphi_*:\pi_1(M)\to\pi_1(|\calN|)$ by Proposition \ref{fiber:prop}. Gromov's Mapping Theorem \cite{Gro} says that an isomorphism on fundamental groups yields an isomorphism and isometry on bounded cohomology, and hence $||\varphi_*(\alpha)||=||\alpha||$ for every cycle $\alpha\in H_*(M)$. In particular, 
$$||M|| = ||[M]|| = ||\varphi_*([M])||.$$
Proposition \ref{efficient:cycle:prop} shows a cycle that represents $\varphi_*([M])$ whose norm equals the number $c(M)$ of vertices of $P$. Therefore $||M||\leqslant c(M)$.
\end{proof}
Theorem \ref{norm:teo} is actually the only tool we have to prove the following.
\begin{cor}
There are manifolds of arbitrarily high complexity, in any dimension $n$.
\end{cor}

\section{Riemannian geometry} \label{riemannian:section}
\subsection{Hyperbolic manifolds}
Theorem \ref{zero:norm:teo} implies the following.

\begin{cor} \label{interior:cor}
Let $M$ be a compact manifold whose interior admits a complete hyperbolic metric of finite volume. Then $c(M)>0$.
\end{cor}
\begin{proof}
Each boundary component $N^{n-1}\subset \partial M^n$ corresponds to a cusp of $M$ and is a flat $(n-1)$-manifold. By Bieberbach Theorem $\pi_1(N)$ has a finite-index subgroup isomorphic to $\matZ^{n-1}$ and is hence amenable. We have $||M||>0$ because $M$ admits a complete hyperbolic metric \cite{Gro} and therefore $c(M)>0$ by Theorem \ref{zero:norm:teo}.
\end{proof}
This result is sharp since the Gieseking 3-manifold (non-orientable with a Klein bottle cusp) has complexity 1, see \cite{CaHiWe}. 

\begin{rem} We do not know if Corollary \label{interior:cor} holds for hyperbolic manifolds with geodesic boundary; it does in dimensions 2 and 3 (because every compact irreducible, $\partial$-irreducible, and anannular 3-manifold has positive complexity \cite{Mat}).
\end{rem}

Theorem \ref{norm:teo} in turn implies the following finiteness result.
\begin{cor} \label{finitely:many:cor}
For every $n$ and $k$ there are finitely many closed hyperbolic manifolds $M^n$ with $c(M^n)<k$.
\end{cor}
\begin{proof}
In dimension $n = 3$, there are finitely many irreducible manifolds of any given complexity \cite{Mat}, so we are done. In dimension $n\neq 3$ there are finitely many hyperbolic manifolds (up to homeomorphism) of bounded volume. The volume of a hyperbolic manifold is proportional to its Gromov norm. Since $\pi_1(M)$ is torsion-free, Theorem \ref{norm:teo} gives $||M||\leqslant c(M)$ and we are done. \end{proof}
Note that the same assertion for Gromov norm (or equivalently hyperbolic volume) is not true for $n=3$. 

\begin{rem}
We do not know if Corollary \ref{finitely:many:cor} holds for hyperbolic manifolds with cusps and/or geodesic boundary: it does in dimensions 2 and 3 (because there are finitely many compact irreducible, $\partial$-irreducible, and anannular 3-manifolds having a fixed complexity \cite{Mat}).
\end{rem}

\subsection{Manifolds of non-positive curvature}
A closed riemannian manifold with non-positive sectional curvature is aspherical because of Cartan-Hadamard's Theorem. Corollary \ref{aspherical:cor} then implies the following.

\begin{cor} \label{sectional:cor}
A closed riemannian manifold $M$ with non-positive sectional curvature has $c(M)>0$.
\end{cor}
\begin{rem}
In contrast with the hyperbolic case,
Corollary \ref{sectional:cor} does not hold for every compact manifold $M$ whose interior admits a complete metric of non-positive (or even negative) curvature and finite volume (see Corollary \ref{interior:cor}). For instance, a product of a closed surface and a bounded surface, both with $\chi <0$, admits such a metric and has complexity zero by Theorem \ref{product:teo}.
\end{rem}

\subsection{Geometric invariants}
Complexity is related to other interesting geometric invariants. Let $(M,g)$ be a riemannian manifold. The \emph{volume entropy} of $(M,g)$ is the limit
$$\lambda (M,G) = \lim_{R\to\infty} \frac{\log{\rm Vol}\big((B(p,R)\big)} {R}$$
where ${\rm Vol}(B(p,R))$ is the volume of the ball of radius $R$ around a point $p$ in the universal cover $\widetilde M$, taken with respect to the lifted metric $\tilde g$. Such a quantity does not depend on $p$, see \cite{Man}. The \emph{systole} $L(M,g)$ is the length of the shortest closed geodesic which is not homotopically trivial. 

The \emph{volume entropy} $\lambda(M)$ and \emph{systolic constant} $\sigma(M)$ of $M$ are defined respectively as the infimum of the volume entropies and systoles among all metrics $g$ on $M$ of volume 1. The \emph{spherical volume} $T(M)$ is defined by Besson, Courtois, and Gallot in \cite{BCG}.

\begin{teo} \label{geometric:teo}
Let $M$ be a closed orientable manifold with virtually torsion-free infinite fundamental group. If $c(M)=0$ then 
$$T(M)=\lambda(M)=\sigma(M)=0.$$
\end{teo}
\begin{proof}
The manifold $M$ is not essential by Theorem \ref{essential:teo}. Babenko showed in \cite{Ba} that a non-essential manifold $M$ has $\lambda(M)=\sigma(M)=0$. Moreover $\lambda(M)=0$ implies $T(M)=0$, see \cite{BCG}.
\label{check referenza}
\end{proof}
The hypothesis on $\pi_1(M)$ is necessary for the vanishing of $\sigma$, since $c(\matRP^n)=0$ and $\sigma(\matRP^n)>0$. We do not know if it is necessary for the vanishing of $T$ and $\lambda$. It is not necessary for the vanishing of $||\cdot ||$ because of Theorem \ref{zero:norm:teo}.

\subsection{Thin riemannian manifolds} \label{thin:subsection}
The cut locus of a riemannian manifold is sometimes a simple spine. Alexander and Bishop \cite{AleBi, AleBi2} proved that the cut locus of a \emph{thin} riemannian manifold is a simple spine without vertices:  a thin riemannian manifold has therefore complexity zero. This happens for instance if we assign a product metric to $M\times [0,1]$ which is very small on $[0,1]$.

Following~\cite{AleBi}, a riemannian manifold $M$ is \emph{thin} if the 
radii of all metric balls in $M$ are small relative to $K_M$ and 
$\kappa_{\partial M}$, where $K_M$ is the sectional curvature of the interior 
and $\kappa_{\partial M}$ is the normal curvature of the boundary $\partial M$. 
To get a scale-free measure of the width of a manifold, Alexander
and Bishop used the curvature-normalized inradius 
$$J\cdot \max{\{\sup{\sqrt{|K_M|},}\ \sup{ |\kappa_{\partial M}|}\}},$$ 
where $J$ is the supremum over all points in $M$ of 
their distances to $\partial M$.

\begin{teo} [Alexander and Bishop \cite{AleBi}] \label{thin:teo}
There are universal constants $a_2<a_3<\ldots$ such that if a bounded riemannian manifold $M^n$
has (curvature-normalized) inradius less than $a_n$, then $c(M^n)=0$.
\end{teo}
\begin{proof}
As shown in~\cite{AleBi}, 
there exists a sequence of universal constants $a_2<a_3<\dots $
such that if the curvature-normalized inradius of $M$ is less than $a_h$
and $h\leqslant n+1$, 
then the cut locus of $M$ is simple and every point is of type $\geqslant n+1-h$. Therefore it has no vertices for $h<n+1$. The cut locus is a spine, hence $c(M)=0$ in that case.
\end{proof}

\section{Four-manifolds} \label{four:manifolds:section}
We describe here some families of 4-manifolds having complexity zero. The various results proved in this paper show that the set of all 4-manifolds of complexity zero contains all products $N\times N'$ with non-empty boundary or $N\in\{S^2,S^3\}$ and is closed under connected sums, finite coverings, addition of handles of index $\neq 3$, and surgery along simple closed curves.

We show that Fintushel-Stern's infinitely many exotic K3's and Gompf's symplectic manifolds with arbitrary $\pi_1$ are constructed by adding handles of index $\neq 3$ to a product: they thus have complexity zero. The double of a 2-handlebody has also complexity zero, because it is obtained by surgerying along a link in a connected sum of some copies of $S^3\times S^1$.

\subsection{Arbitrary fundamental group}
Closed 4-manifolds with complexity zero may have arbitrary (finitely presented) fundamental groups.
\begin{teo} \label{double:teo}
The double $DM$ of a four-manifold $M$ made of only $0$-, $1$-, and $2$-handles has $c(DM)=0$.
\end{teo}
\begin{proof}
The manifold $M$ can be decomposed into 0-, 1-, and 2-handles. The 0- and 1-handles form a 1-handlebody $H$. The cores of the 2-handles are attached along a link $L\subset\partial H$. The double $DM$ is the result of a surgery along $L$ in $DH$. Now $DH$ is the connected sum of some copies of $S^1\times S^3$ and hence $c(DH)=0$ by Corollary \ref{sphere:product:cor} and Theorem \ref{sum:teo}. Then $c(DM)=0$ by Corollary \ref{surgery:cor}.
\end{proof}
\begin{cor} \label{every:cor}
Every finitely presented group is the fundamental group of a closed orientable 4-manifold of complexity zero.
\end{cor}
\begin{proof}
Every finitely presented group is the fundamental group of a 4-manifold made of 0-, 1-, and 2-handles. Doubling it does not change the fundamental group.
\end{proof}
\begin{rem} It is not true that every double has complexity zero. For instance, if $M$ is a product of a closed surface and a bounded surface, both with $\chi\leqslant 0$, then $DM$ is a product of closed aspherical surfaces and hence $c(DM)>0$ by Corollary \ref{aspherical:cor}. Note that $c(M)=0$ by Theorem \ref{product:teo}.
\end{rem}

\subsection{Capping with elliptic surfaces}

\begin{figure}
 \begin{center}
 \includegraphics[width = 8cm]{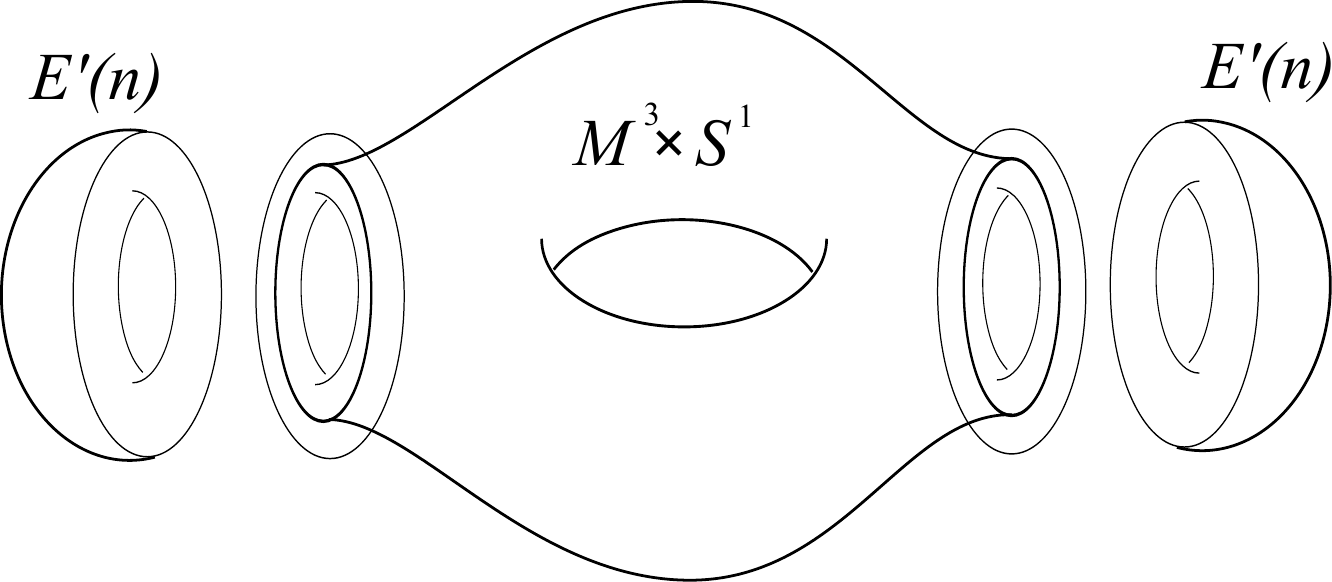}
 \end{center}
\caption{A manifold constructed by taking a product $M^3\times S^1$ with $\partial M^3$ consisting of some tori and capping the boundary components with some copies of $E'(n)$ has complexity zero. Among these manifolds we find infinitely many exotic $K3$ surfaces (Theorem \ref{K3:teo}) and symplectic manifolds with arbitrary fundamental group (Theorem \ref{symplectic:teo}).}
\label{cap:fig}
\end{figure}

Consider the elliptic surface $E(n)$, described as a fiber-connected sum of $n$ copies of $E(1)$ as in \cite{GoSti}. The elliptic surface has an elliptic fibration $E(n)\to S^2$, whose regular fiber is a torus. The regular neighborhood $R(T)$ of a regular fiber $T$ is homeomorphic to a product $T\times D^2$. Define 
$$E'(n) = E(n) \setminus \interior{R(T)}.$$
This is a simply connected manifold with a 3-torus as its boundary. It has a decomposition with only 0- and 2-handles \cite[Figg 8.9 and 8.10]{GoSti}. This manifold is sometimes used to cap some boundary component of a 4-manifold, see Fig~\ref{cap:fig}. We have the following.

\begin{lemma} \label{capping:lemma}
Let $M^3$ be a compact 3-manifold whose boundary consists of some tori. Let $N^4$ be obtained by capping some boundary components of $M^3\times S^1$ via some copies of $E'(n)$, along some maps. Then $c(N)=0$.
\end{lemma}
\begin{proof}
We have $c(M\times S^1)=0$ by Theorem \ref{product:teo}. The 0- and 2-handles of $E'(n)$ transform into 4- and 2-handles when attached to $\partial (M\times S^1)$: since there is no 3-handle, we have $c(N)=0$.
\end{proof}

\subsection{Simply connected manifolds}
We start with the following consequence of Theorem \ref{handles:teo}.
\begin{cor} A closed 4-manifold admitting a handle decomposition without 3-handles has complexity zero.
\end{cor}
Every such manifold is necessarily simply connected. However, it is still unknown whether every simply connected 4-manifold has a handle decomposition without 3-handles (or even without 1- and 3-handles). 

For every knot $K\subset S^3$, Fintushel and Stern constructed an exotic K3 surface $X_K$ whose Seiberg-Witten invariant is roughly the Alexander polynomial of the knot \cite{FiSt}. Among them there are infinitely many distinct exotic K3 surfaces.
\begin{teo} \label{K3:teo}
The manifold $X_K$ has complexity zero for every knot $K$.
\end{teo}
\begin{proof}
The manifold $X_K$ is constructed by capping $(S^3\setminus R(K))\times S^1$ with one copy of $E'(2)$, see \cite{FiSt}. The result follows from Lemma \ref{capping:lemma}.
\end{proof}
We do not know if $X_K$ admits a handle decomposition without 3-handles.

\subsection{Symplectic manifolds}
Gompf constructed in \cite{Go} a family of closed symplectic 4-manifolds with arbitrary fundamental group. It turns out that these manifolds have complexity zero. We can therefore strengthen Corollary \ref{every:cor}.

\begin{teo} \label{symplectic:teo}
Every finitely presented group is the fundamental group of a closed symplectic $4$-manifold of complexity zero. The manifold can be chosed to be spin or nonspin.
\end{teo}
\begin{proof}
Gompf's construction starts with a particular 3-manifold $M^3$ bounded by some tori, and hence caps the boundary components of $M^3\times S^1$ with copies of $E'(n)$, see \cite{Go}. The result follows from Lemma \ref{capping:lemma}.
\end{proof}

\end{document}

%% file: defs_complexity.tex
\newtheorem{lemma}{Lemma}[section]
\newtheorem{teo}[lemma]{Theorem}
\newtheorem{prop}[lemma]{Proposition}
\newtheorem{cor}[lemma]{Corollary}

\theoremstyle{definition}
\newtheorem{defn}[lemma]{Definition}

\newtheorem{example}[lemma]{Example}
\newtheorem{ex}[lemma]{Exercise} 

\theoremstyle{remark}
\newtheorem{rem}[lemma]{Remark}

\newcommand{\matR} {\ensuremath {\mathbb{R}}}

\newcommand{\matZ} {\ensuremath {\mathbb{Z}}}

\newcommand{\matH} {\ensuremath {\mathbb{H}}}

\newcommand{\matRP} {\ensuremath {\mathbb{RP}}}
\newcommand{\matCP} {\ensuremath {\mathbb{CP}}}

\newcommand{\calN} {\ensuremath {\mathcal{N}}}

\newcommand{\calM} {\ensuremath {\mathcal{M}}}
\newcommand{\calC} {\ensuremath {\mathcal{C}}}

\newcommand{\calP} {\ensuremath {\mathcal{P}}}

\newcommand{\interior}[1]{{\rm int}(#1)}
\newfont{\Got}{eufm10 scaled 1200}

\newcommand{\lk}{{\rm lk}}
\newcommand{\st}{{\rm st}}

\newcommand{\cref}{c^{\rm ref}}

%% file: model.pdftex_t
\begin{picture}(0,0)%
\includegraphics{model.pdftex}%
\end{picture}%
\setlength{\unitlength}{4144sp}%
\begingroup\makeatletter\ifx\SetFigFont\undefined%
\gdef\SetFigFont#1#2#3#4#5{%
  \reset@font\fontsize{#1}{#2pt}%
  \fontfamily{#3}\fontseries{#4}\fontshape{#5}%
  \selectfont}%
\fi\endgroup%
\begin{picture}(3543,2180)(399,-3849)
\put(3466,-2806){\makebox(0,0)[lb]{\smash{{\SetFigFont{12}{14.4}{\rmdefault}{\mddefault}{\updefault}{\color[rgb]{0,0,0}$n=3$}%
}}}}
\put(1486,-2806){\makebox(0,0)[lb]{\smash{{\SetFigFont{12}{14.4}{\rmdefault}{\mddefault}{\updefault}{\color[rgb]{0,0,0}$n=1$}%
}}}}
\put(2476,-2806){\makebox(0,0)[lb]{\smash{{\SetFigFont{12}{14.4}{\rmdefault}{\mddefault}{\updefault}{\color[rgb]{0,0,0}$n=2$}%
}}}}
\put(399,-2188){\makebox(0,0)[lb]{\smash{{\SetFigFont{12}{14.4}{\rmdefault}{\mddefault}{\updefault}{\color[rgb]{0,0,0}$\partial\Pi^n:$}%
}}}}
\put(405,-3433){\makebox(0,0)[lb]{\smash{{\SetFigFont{12}{14.4}{\rmdefault}{\mddefault}{\updefault}{\color[rgb]{0,0,0}$\Pi^n$:}%
}}}}
\end{picture}%

%% file: simple.pdftex_t
\begin{picture}(0,0)%
\includegraphics{simple.pdftex}%
\end{picture}%
\setlength{\unitlength}{4144sp}%
\begingroup\makeatletter\ifx\SetFigFont\undefined%
\gdef\SetFigFont#1#2#3#4#5{%
  \reset@font\fontsize{#1}{#2pt}%
  \fontfamily{#3}\fontseries{#4}\fontshape{#5}%
  \selectfont}%
\fi\endgroup%
\begin{picture}(6539,1435)(226,-3614)
\put(226,-2311){\makebox(0,0)[lb]{\smash{{\SetFigFont{12}{14.4}{\rmdefault}{\mddefault}{\updefault}{\color[rgb]{0,0,0}$n=1$:}%
}}}}
\put(3466,-2311){\makebox(0,0)[lb]{\smash{{\SetFigFont{12}{14.4}{\rmdefault}{\mddefault}{\updefault}{\color[rgb]{0,0,0}$n=2$:}%
}}}}
\end{picture}%

%% file: bubbling.pdftex_t
\begin{picture}(0,0)%
\includegraphics{bubbling.pdftex}%
\end{picture}%
\setlength{\unitlength}{4144sp}%
\begingroup\makeatletter\ifx\SetFigFont\undefined%
\gdef\SetFigFont#1#2#3#4#5{%
  \reset@font\fontsize{#1}{#2pt}%
  \fontfamily{#3}\fontseries{#4}\fontshape{#5}%
  \selectfont}%
\fi\endgroup%
\begin{picture}(12205,4576)(204,-3924)
\put(1216,-3481){\makebox(0,0)[lb]{\smash{{\SetFigFont{41}{49.2}{\rmdefault}{\mddefault}{\updefault}{\color[rgb]{0,0,0}$K$}%
}}}}
\put(4141,-61){\makebox(0,0)[lb]{\smash{{\SetFigFont{41}{49.2}{\rmdefault}{\mddefault}{\updefault}{\color[rgb]{0,0,0}$Q$}%
}}}}
\put(11161,-61){\makebox(0,0)[lb]{\smash{{\SetFigFont{41}{49.2}{\rmdefault}{\mddefault}{\updefault}{\color[rgb]{0,0,0}$P$}%
}}}}
\end{picture}%

%% file: buca.pdftex_t
\begin{picture}(0,0)%
\includegraphics{buca.pdftex}%
\end{picture}%
\setlength{\unitlength}{4144sp}%
\begingroup\makeatletter\ifx\SetFigFont\undefined%
\gdef\SetFigFont#1#2#3#4#5{%
  \reset@font\fontsize{#1}{#2pt}%
  \fontfamily{#3}\fontseries{#4}\fontshape{#5}%
  \selectfont}%
\fi\endgroup%
\begin{picture}(8123,1100)(2243,-4034)
\put(4932,-4003){\makebox(0,0)[lb]{\smash{{\SetFigFont{8}{9.6}{\rmdefault}{\mddefault}{\updefault}{\color[rgb]{0,0,0}$\gamma$}%
}}}}
\end{picture}%